\documentclass[a4paper,12pt]{article}

\usepackage{amsfonts,amsthm,amssymb,amsmath}
\usepackage{hyperref}
\usepackage{fourier}

\newtheorem{theorem}{Theorem}[section]
\newtheorem{proposition}[theorem]{Proposition}

\newtheorem{lemma}[theorem]{Lemma}
\theoremstyle{definition}

\newtheorem{remark}[theorem]{Remark}

\newcommand{\jpn}{\mathcal J(p,n)}
\newcommand{\jmn}{\mathcal J(m,n)}

\newcommand{\jbn}[1]{\mathcal J(#1,n)}
\newcommand{\bi}{\mathbf i}
 \newcommand{\bj}{\mathbf j}

\DeclareMathOperator{\mon}{mon}
\DeclareMathOperator{\card}{card}
\DeclareMathOperator{\re}{Re}

\title{Multipliers of Dirichlet series and monomial series expansions of holomorphic functions in infinitely many variables}

\author{
Fr\'ed\'eric Bayart \footnote{Laboratoire de Math\'ematiques
Universit\'e Blaise Pascal Campus des C\'ezeaux, F-63177 Aubi\`ere Cedex (France)
}
\and Andreas Defant\footnote{Institut f\"ur Mathematik. Universit\"at Oldenburg. D-26111 Oldenburg (Germany) }
\and Leonhard Frerick\footnote{Fachbereich IV - Mathematik, Universit\"{a}t Trier, D-54294 Trier (Germany)}
\and Manuel Maestre\footnote{Dep. An\'alisis Matem\'atico. Fac. Matem\'aticas. Universidad de Valencia. 46100 Burjassot (Spain)} \and
Pablo Sevilla-Peris  \footnote{IUMPA. Universitat Polit\`ecnica de Val\`encia. 46022 Val\`encia (Spain)}}
\date{}

\begin{document}
\maketitle

\footnotetext{The  second, fourth and fifth  authors were  supported  by  MICINN and FEDER Project MTM2011-22417. The fourth author was
also supported by PrometeoII/2013/013. The fifth
author was also supported by project SP-UPV20120700.}

\vspace{-8mm}

\begin{abstract}
\noindent
\noindent Let $\mathcal{H}_\infty$ be the  set of all ordinary Dirichlet series $D=\sum_n a_n n^{-s}$
representing bounded holomorphic functions on the right half plane. A multiplicative sequence $(b_n)$
of complex numbers is said to be an $\ell_1$-multiplier for  $\mathcal{H}_\infty$   whenever  $\sum_n |a_n b_n| < \infty$
for every  $D \in \mathcal{H}_\infty$.
We study the problem of describing such sequences
  $(b_n)$ in terms of the asymptotic decay of  the subsequence $(b_{p_j})$, where  $p_j$ denotes  the $j$th prime number. Given a multiplicative sequence $b=(b_n)$ we prove (among other results): $b$ is an
$\ell_1$-multiplier for $\mathcal{H}_\infty$ provided $|b_{p_j}| < 1$ for all $j$ and
$\overline{\lim}_n \frac{1}{\log n} \sum_{j=1}^n b_{p_j}^{*2} < 1$, and conversely,
if $b$ is an $\ell_1$-multiplier  for $\mathcal{H}_\infty$,  then $|b_{p_j}| < 1$ for all $j$ and $\overline{\lim}_n \frac{1}{\log n} \sum_{j=1}^n b_{p_j}^{*2} \leq  1$ (here $b^*$ stands for the decreasing rearrangement of $b$).
    Following an ingenious idea of Harald Bohr it turns out that this problem is intimately related with the question of characterizing those sequences $z$ in the infinite dimensional polydisk $\mathbb{D}^\infty$ (the open unit ball of $\ell_\infty$) for which  every
bounded and holomorphic function $f$ on $\mathbb{D}^\infty$ has an absolutely convergent monomial series expansion $\sum_{\alpha} \frac{\partial_\alpha f(0)}{\alpha!} z^\alpha$. Moreover, we study  analogous problems in Hardy spaces of Dirichlet series and Hardy spaces of functions on the infinite dimensional polytorus $\mathbb{T}^\infty$.
\end{abstract}
\pagebreak

\tableofcontents

\section{Introduction}

Recall from  \cite{DeFrOrOuSe11} that the  precise asymptotic order of the Sidon constant
of all finite Dirichlet polynomials $\sum_{n=1}^N a_n n^{-s}$ is given by
\begin{equation} \label{ojeoje}
\displaystyle
\sup_{a_1, \ldots, a_N \in \mathbb{C}} \,\,\,\frac{\sum_{n=1}^N |a_n|}
{\sup_{t \in \mathbb{R}} \big| \sum_{n=1}^N a_n n^{-it}   \big|}
\,\,\,=\,\,\,
\frac
{\sqrt{N}}
{e^
{
\frac{1}{\sqrt{2}} \big(1+ o(1)\big)
\sqrt{\log N  \log \log N}
}
}
\,\,\,.
\end{equation}
This result has its origin in fundamental works of Hilbert \cite{Hi09},  Bohr \cite{Bo13_Goett}, Toeplitz \cite{To13} and Bohnenblust-Hille
\cite{BoHi31}, and it is the  final outcome of a long series of results due to
\cite{BaCaQu06,Br08,KoQu01,Qu95,QQ13}.

As usual we denote by $\mathcal{H}_{\infty}$  the vector space of all ordinary Dirichlet series $\sum_n a_n n^{-s}$
representing bounded holomorphic functions on the right half plane
(which  together with  the sup norm  forms a Banach space).
Applying \eqref{ojeoje} to dyadic blocks,
it was  proved in \cite{DeFrOrOuSe11} (completing  earlier results from \cite{BaCaQu06}) that the
supremum over all $c >0$ for which
\begin{equation} \label{oje}
\sum_{n=1}^\infty |a_n|\,
\frac{
e^{ c\sqrt{\log n\log \log n}}}{n^\frac{1}{2}}
<\infty \,\,\, \text{ for all  } \,\,\, \sum_n a_n n^{-s} \in \mathcal{H}_{\infty}
\end{equation}
equals $1/\sqrt{2}$. In other terms, all  sequences $$(b_n)=\big(e^{ (1/\sqrt{2}-\varepsilon)\sqrt{\log n\log \log n}} n^{-1/2}\big)\,,  \,\,\, 0 < \varepsilon < 1/\sqrt{2}$$
are   $\ell_1$-multiplier of $\mathcal{H}_{\infty}$ in the sense that
\begin{equation} \label{multtii}
\sum_{n=1}^\infty |a_nb_n| < \infty\,\,\, \text{ for all  } \,\,\, \sum_n a_n n^{-s} \in \mathcal{H}_{\infty}.
\end{equation}
Recall  that a sequence $(b_n)$ is said to be  (completely) multiplicative whenever $b_{nm} = b_n b_m$ for all $n,m$, and \eqref{multtii}
obviously shows that the sequence $(1/\sqrt{n})$ is a multiplicative  $\ell_1$-multiplier of $\mathcal{H}_{\infty}$.
 Clearly, there are  more such multiplicative $\ell_1$-multipliers of $\mathcal{H}_{\infty}$. For example, it will turn out that all multiplicative sequences $(b_n)$ with $|b_n|< 1$ for  all $n$ and such
that $b_{p_j}=0$ for all but finitely many  $j$  have this property; here as usual $p = (p_j) = \{2, 3,5 \ldots\}$ stands for the sequence of primes.

In this article we  intend to study the problem of describing all multiplicative $\ell_1$-multipliers $(b_n)$ of $\mathcal{H}_{\infty}$
   in terms of the asymptotic decay of  the subsequence $(b_{p_j})$.

Surprisingly, this question is intimately related with the following  natural problem: Do $\mathbb{C}$-valued
holomorphic functions  on the infinite dimensional polydisk $B_{\ell_\infty}$
(the open unit ball  of the Banach space $\ell_{\infty}$ of all bounded scalar sequences), like in finite dimensions, have a  reasonable monomial series expansion?

The crucial link is due to a  genius observation of Harald Bohr from \cite{Bo13_Goett} which we explain now:
Denote by $\mathfrak{P}$  the vector space  of all formal power series
$\sum_{\alpha} c_{\alpha} z^{\alpha}$,
and by $\mathfrak{D}$ the vector space  of all Dirichlet series $\sum a_{n} n^{-s}$.
By the fundamental theorem of arithmetics each $n \in \mathbb{N}$ has a unique prime decomposition
 $n= p^{\alpha}=p_{1}^{\alpha_{1}} \cdots p_{k}^{\alpha_{k}}$ with a multiindex $\alpha \in \mathbb{N}_{0}^{(\mathbb{N})}$
(i.e., $\alpha$ is a finite  sequences
of elements $\alpha_k \in \mathbb{N}_{0}$).
Then  the so-called Bohr transform $\mathfrak{B}$ is a linear algebra homomorphism:
\begin{align} \label{vision}
\mathfrak{B} : \mathfrak{P} & \longrightarrow \mathfrak{D} \,\,, \,\,\,\,\,\,
\textstyle\sum_{\alpha \in \mathbb{N}_{0}^{(\mathbb{N})}} c_{\alpha} z^{\alpha}
\rightsquigarrow \textstyle\sum_{n=1}^{\infty} a_{n} n^{-s}\,\,\,\text{ with  }\,\,\ a_{p^{\alpha}} = c_{\alpha}\,.
\end{align}
Hilbert in \cite{Hi09} was among the very first who started a systematic study  of the concept of analyticity
for functions in infinitely many
variables. According to Hilbert, an analytic function in infinitely many variables is
a $\mathbb{C}$-valued function defined on the infinite dimensional polydisk   $B_{\ell_\infty}$ (see above)
which  has a pointwise convergent monomial series expansion:
\begin{equation}\label{monexp}
f(z)=\sum_{\alpha\in \mathbb{N}_{0}^{(\mathbb{N})}} c_{\alpha} z^{\alpha}\,\,, \,\,\,\,z \in B_{\ell_\infty} \,.
\end{equation}
In \cite{Hi09} (see also \cite[p.~65]{Hilbert_Gesam_3}) he gave the following criterion for a formal power series $\sum_{\alpha} c_{\alpha} z^{\alpha}$ to generate such a function (i.e.,
to converge absolutely at each
point of $B_{\ell_\infty}$): Every $k$-dimensional section $\sum_{\alpha\in \mathbb{N}_{0}^{k}}
c_{\alpha} z^{\alpha}$ of the series is pointwise convergent on $\mathbb{D}^{k}$,
and moreover
\begin{equation}\label{e}
\sup_{k \in \mathbb{N}} \sup_{z \in \mathbb{D}^{k}} \Big|\sum_{\alpha \in \mathbb{N}_{0}^{k}} c_{\alpha} z^{\alpha}\Big| < \infty\;.
\end{equation}
But this criterion is not correct as was later discovered by Toeplitz (see below \eqref{Toeppi}). Why?

 Today a holomorphic function $f:B_{\ell_\infty} \rightarrow \mathbb{C}$
is nothing else than a Fr\'{e}chet differentiable function  $f: B_{\ell_\infty} \rightarrow \mathbb{C}$.
As usual the Banach space of all bounded holomorphic $f:  B_{\ell_\infty} \rightarrow \mathbb{C}$
endowed with the supremum norm will be denoted by $H_\infty(B_{\ell_\infty})$.
Important examples of such functions are bounded $m$-homo\-geneous polynomials
$P: \ell_{\infty} \rightarrow \mathbb{C}$, restrictions of bounded $m$-linear forms
on $\ell_{\infty} \times \cdots \times \ell_{\infty}$ to the diagonal. The vector space $\mathcal{P}(^m \ell_{\infty})$ of all  such $P$ together with the norm $\|P\| = \sup_{z \in B_{\ell_\infty}} |P(z)|$ forms a closed subspace of $H_\infty(B_{\ell_\infty})$.

From the theory in finitely many variables it is well known that  every holomorphic $\mathbb{C}$-valued mapping $f$ on the $k$-dimensional polydisk $\mathbb{D}^{k}$ has a
monomial  series expansion which converges to $f$ at every point of $\mathbb{D}^{k}$.
More precisely, for every such $f$ there is a unique family
$\big(c_{\alpha}(f)\big)_{\alpha\in \mathbb{N}_{0}^{k}}$ in $\mathbb{C}$ such that $f (z) =  \sum_{\alpha \in \mathbb{N}_{0}^{k}} c_{\alpha}(f) z^{\alpha}$
for every $z \in \mathbb{D}^{k}$. The coefficients can be calculated through the Cauchy integral formula
or partial derivatives: For each $0 < r < 1$ and each $\alpha$
\begin{equation} \label{monomialcoefficients}
c_{\alpha}(f) = \frac{\partial^{\alpha} f (0)}{\alpha !} = \frac{1}{(2 \pi i)^{k}} \int_{|z_{1}| = r}\ldots \int_{|z_k| = r}
\frac{f (z)}{z^{\alpha + 1}} dz_{1} \ldots dz_k.
\end{equation}
Clearly, every
holomorphic function $f:B_{\ell_\infty} \rightarrow \mathbb{C}$, whenever restricted
to a finite dimensional section $\mathbb{D}^{k}=\mathbb{D}^{k} \times \{0\}$, has  an everywhere convergent  power
series expansion $\sum_{\alpha \in \mathbb{N}_{0}^{k}} c^{(k)}_{\alpha}(f)
z^{\alpha}$, $z \in \mathbb{D}^{k}$. And
from  \eqref{monomialcoefficients} we  see that
$c_{\alpha}^{(k)}(f) = c_{\alpha}^{(k+1)}(f)$ for $\alpha \in \mathbb{N}_{0}^{k} \subset \mathbb{N}_{0}^{k+1}$. Thus again there is a unique family
$\big(c_{\alpha}(f)\big)_{\alpha \in \mathbb{N}_{0}^{(\mathbb{N})}}$ in $\mathbb{C}$ such that at least for all $k\in \mathbb{N}$
and all  $z \in \mathbb{D}^{k}$
\[
 f (z) = \sum_{\alpha\in \mathbb{N}_{0}^{(\mathbb{N})}} c_{\alpha}(f) z^{\alpha}\,.
\]
This power series  is called the monomial series expansion of $f$, and $c_{\alpha}= c_{\alpha}(f)$ are its monomial coefficients; by definition they satisfy
\eqref{monomialcoefficients} whenever $\alpha\in \mathbb{N}_{0}^{k}$.

At first one could expect that each $f \in H_\infty(B_{\ell_\infty})$ has a monomial series expansion which again converges at every point
and represents the function. But this is not the case: Just take a non-zero functional on $\ell_{\infty}$
that is $0$ on $c_{0}$ (the space of null sequences); by definition, its monomial series expansion is $0$  and clearly does not represent the function. Moreover, since such  a functional  obviously satisfies \eqref{e}, although it is not analytic (in Hilbert's sense), the  criterion of Hilbert turns out to be false.

In order to avoid this example one
could now  try with  the open unit ball $B_{c_0}$  of $c_{0}$ instead of $B_{\ell_\infty}$. But Hilbert's criterion remains false:
Note first that a simple extension argument (see e.g. \cite[Lemma~2.2]{DeGaMa10})
allows to identify all formal power series  satisfying \eqref{e}  with  all bounded holomorphic functions on $B_{c_0}$; more precisely,
each $f \in H_\infty(B_{c_0})$ has a monomial series expansion as in \eqref{e}, and conversely each
power series satisfying \eqref{e} gives rise to  a unique $f \in H_\infty(B_{c_0})$ for which  $c_\alpha = c_\alpha(f)$
for all $\alpha$.

But then
\eqref{e} does not imply \eqref{monexp} since by an example of
Toeplitz from \cite{To13} there is a $2$-homogeneous bounded polynomial $P$ on $c_0$
such that
\begin{align} \label{Toeppi}
\forall \, \varepsilon >0 \,
\exists\, x \in \ell_{4+\varepsilon} \,:
\,\sum_\alpha |c_\alpha(P) x^{\alpha}| = \infty \, .
\end{align}
This means that there are functions $f \in H_\infty(B_{c_0})$ that cannot be pointwise described by their
monomial series expansions as in \eqref{monexp} which, at least at first glance, seems disillusioning.
Indeed, this fact in infinite
dimensions produces a sort of  dilemma: There is no way to develop a complex analysis of functions in infinitely many  variables which simultaneously
handles phenomena on differentiability and analyticity (as it happens  in finite dimensions).

One of the main advances of this article is to give an almost complete description of what we call the set of monomial convergence of all bounded holomorphic functions on the open unit ball $B_{\ell_\infty}$ of $\ell_\infty$:
\begin{equation} \label{josquin}
\mon H_\infty(B_{\ell_\infty}) =
\Big\{z\in B_{\ell_\infty}\,\, \big|\,\, \forall \, f\in H_\infty(B_{\ell_\infty} ) :\; f(z) = \sum_{\alpha\in \mathbb{N}_{0}^{(\mathbb{N})}}
c_\alpha(f)z^{\alpha} \Big \}\,.
\end{equation}
We recall that the decreasing rearrangement of  $z \in \ell_{\infty}$ is given  by
\[
z^{*}_{n} = \inf \big\{ \sup_{j \in \mathbb{N} \setminus J } | z_{j} |  \,\, \big|\,\,  J \subset \mathbb{N} \  , \ \card (J) < n \big\} \,\,,
\]
and use it to define the set
\[
\displaystyle
\mathbf{B} = \Big\{ z \in B_{\ell_\infty} \,\, \,\big|\,\,\, \boldsymbol{b} (z)=\Big( \limsup_{n \rightarrow \infty}  \frac{1}{\log n} \sum_{j=1}^{n} z^{* 2}_{j}\Big)^{1/2}  < 1 \Big\}\,.
\]
Then our main result is  Theorem~\ref{Leonhard} that shows
\begin{align} \label{reformel}
 \mathbf{B} \,\,\,\,\subset \,\,\,\,\mon H_{\infty}(B_{\ell_\infty}) \,\,\,\,\subset\,\,\,\, \mathbf{\overline{B}}\,.
\end{align}
As we intend to indicate in the following sections, this  result  has a long list of forerunners
(due to various authors, see e.g.
   \cite{BoHi31,Bo13_Goett, Bo13,CoGa86,DeMaPr09,Hilbert_Gesam_3,Hi09,To13}).
 In  \eqref{11}, \eqref{22}, \eqref{33}, \eqref{crelle}, \eqref{gadafi} as well as  \eqref{BBHthm},\eqref{jetztdoch}, \eqref{BHthm1}, \eqref{orkan}  it will become clear that
   $\mon H_{\infty}(B_{\ell_\infty})$ was known to be very close to  $\ell_2 \cap B_{\ell_\infty}$.
      But \eqref{reformel} adds a new
level of precision that enables us to extract much more precise information from  monomial convergence  of
holomorphic functions on the infinite dimensional polydisk than before.

This in particular gets  clear if we finally return to the beginning of this introduction -- let us return to the description of all multiplicative
$\ell_1$-multipliers of  $\mathcal{H}_{\infty}$ using Bohr's transform from \eqref{vision}.
The following  fact, essentially due to Bohr \cite{Bo13_Goett} and later rediscovered in  \cite[Lemma~2.3 and Theorem~3.1]{HeLiSe97}, is essential: The Bohr transform $\mathfrak{B}$  induces an isometric algebra isomorphism
from $H_{\infty}(B_{c_0})$ onto $\mathcal{H}_\infty$\,,
\begin{equation} \label{fund}
  H_{\infty}(B_{c_0}) \,\,\,  = \,\,\, \mathcal{H}_{\infty}\,.
\end{equation}
This identification in fact allows to identify the multiplicative $\ell_1$-multipliers of $\mathcal{H}_{\infty}$
with the elements in $\mon H_{\infty}(B_{\ell_\infty})$: Given a sequence $(b_n) \in \mathbb{C}^{\mathbb{N}}$, we have that
\[
\text{$(b_n)$ is an $\ell_1$-multiplier for  $\mathcal{H}_\infty$}
\,\,\,
\Leftrightarrow
\,\,\,
(b_{p_k}) \in \mon H_{\infty}(B_{\ell_\infty})\,.
\]
Observe that  this way we may deduce  from \eqref{oje}  that the sequence
\begin{equation} \label{first}
\big(1/\sqrt{p_k}\big) \in\mon H_{\infty}(B_{\ell_\infty})\,;
\end{equation}
this seems to be the first non-trivial example  which distinguishes  $\mon H_{\infty}(B_{\ell_\infty})$ from
$\ell_2 \cap B_{\ell_\infty}$.
But note that this can also be seen using \eqref{reformel}; indeed, \eqref{first} is a very particular case of the following reformulation of \eqref{reformel} (see Section 4) which
is an almost complete characterization of {\it all} multiplicative $\ell_1$-multipliers for $\mathcal{H}_\infty$. For all multiplicative sequences
$(b_n) \in \mathbb{C}^{\mathbb{N}}$ we have that
\begin{itemize}
\item
$(b_n)$ is an
$\ell_1$-multiplier for $\mathcal{H}_\infty$ provided we have that $|b_{p_j}| < 1$ for all $j$ and
$\boldsymbol{b}\big((b_{p_j})\big) < 1$, and conversely,
\item
if $(b_n)$ is $\ell_1$-multiplier  for $\mathcal{H}_\infty$,  then $|b_{p_j}| < 1$ for all $j$ and $\boldsymbol{b}\big((b_{p_j})\big) \leq  1$.
\end{itemize}
In Section 3 we extend  our concept on sets of monomial convergence  to $H_p$-functions defined on the infinite dimension torus $\mathbb{T}^\infty$ (see \eqref{guillaumedufay} for the precise definition); here $\mathbb{T}$ denotes the torus (the unit circle of $\mathbb{C}$)
and $\mathbb{T}^{\infty}$ the infinite dimensional polytorus (the countable cartesian product of $\mathbb{T}$). The Banach space $H_\infty(B_{c_0})$ can be isometrically  identified with the Banach space $H_\infty(\mathbb{T}^{\infty})$ of  all $L_\infty$-functions $f:\mathbb{T}^{\infty} \rightarrow\mathbb{C}$  with Fourier coefficients $\hat{f}(\alpha)=0$
for $\alpha \in \mathbb{Z}^{(\mathbb{N})} \setminus \mathbb{N}_{0}^{(\mathbb{N})}$; this was proved in \cite{CoGa86} (see also Proposition~\ref{H intftys}). For $1 \leq p \leq \infty$ we define
\[
 \mon H_{p}(\mathbb{T}^{\infty}) = \Big\{ z \in \mathbb{C}^{\mathbb{N}} \,\, \big|\, \,\sum_{\alpha} \vert \hat{f}(\alpha) z^{\alpha} \vert < \infty
\text{ for all }  f \in  H_{p}(\mathbb{T}^{\infty}) \Big\}\,.
\]
Then it is not hard to see that $\mon H_{\infty}(\mathbb{T}^{\infty})= \mon H_{\infty}(B_{\ell_\infty})$, but in contrast to \eqref{reformel}   we have
\[
\mon H_{p}  (\mathbb{T}^{\infty}) = \ell_{2} \cap B_{\ell_\infty}  \,\,\, \text{ for all   }\,\,\, 1 \leq p < \infty\,.
\]
This way we extend and complement results of Cole and Gamelin from \cite{CoGa86}.
Finally, in  Section 4 we use Bohr's vision from \eqref{vision} to interpret all these results on sets of monomial convergence of $H_p$-functions  in terms of multiplicative $\ell_1$-multipliers for  $\mathcal{H}_p$-Dirichlet series (as was already described above for the case $p= \infty$).

\section{Monomial expansion of $\boldsymbol{H}_{\boldsymbol{\infty}}$-functions in infinitely many variable} \label{domains}
Our definition of sets   of monomial convergence \eqref{josquin} has its roots in Bohr's seminal article \cite{Bo13_Goett}, and the first systematic study of such sets was undertaken in \cite{DeMaPr09}. Recall from the introduction that
\[
\mon H_\infty(B_{\ell_\infty}) =
\Big\{z\in B_{\ell_\infty}\,\, \big|\,\, \forall \, f\in H_\infty(B_{\ell_\infty} ) :\; f(z) = \sum_{\alpha\in \mathbb{N}_{0}^{(\mathbb{N})}}
c_\alpha(f)z^{\alpha} \Big \}\,,
\]
and define similarly for $m \in \mathbb{N}$
\[
\mon \mathcal{P}(^m\ell_{\infty}) =
\Big\{
z\in \ell_{\infty} \; \big|\, \forall P \in \mathcal{P}(^m\ell_{\infty}):\; P (z) = \sum_{\alpha\in \mathbb{N}_{0}^{(\mathbb{N})}} c_\alpha(P) z^{\alpha} \Big\}.
\]
Since we here consider functions $f$ defined on $B_{\ell_\infty}$ as well as  polynomials $P$ defined on $\ell_\infty$, we clearly cannot avoid to define
the preceding two sets  as subsets of  $B_{\ell_\infty}$ and $\ell_\infty$, respectively.
Nevertheless we can  give two slight  reformulations
which will be of particular importance  when we translate our forthcoming results into terms of multipliers for Dirichlet series:
\begin{align} \label{heide}
\mon H_\infty(B_{\ell_\infty}) =
\Big\{z\in \mathbb{C}^{\mathbb{N}}\,\, \big|\,\, \forall \, f\in H_\infty(B_{\ell_\infty} ) :\;  \sum_{\alpha\in \mathbb{N}_{0}^{(\mathbb{N})}}
\big|c_\alpha(f)z^{\alpha}\big| < \infty \Big \}
\end{align}
\begin{align} \label{heide1}
\mon \mathcal{P}(^m\ell_{\infty}) =
\Big\{
 z\in \mathbb{C}^{\mathbb{N}}\; \big|\, \forall P \in \mathcal{P}(^m\ell_{\infty}):\; \sum_{\alpha\in \mathbb{N}_{0}^{(\mathbb{N})}} \big|c_\alpha(P)z^{\alpha}\big| < \infty \Big\}
\end{align}
 The argument for these two equalities is short: Denote the set in \eqref{heide1} by $U$, and that in
 \eqref{heide} by $V$. For $z \in U$ it was shown in \cite[p.29-30]{DeMaPr09} that $z \in c_0$. Then an obvious continuity argument gives  the equality in \eqref{heide1}. Take now $z \in V \subset U$. Considering bounded holomorphic functions on the open disk $\mathbb{D}$, we see immediately that $|z_n|<1$ for all $n$. The equality in \eqref{heide}
 again follows by  continuity.

In the above definitions we may replace $\ell_\infty$ by $c_0$.
Davie and Gamelin showed in \cite[Theorem~5]{DaGa89} that every function in $H_\infty(B_{c_0})$ can be extended to a function in
$H_\infty(B_{\ell_\infty})$ with the same norm. Using this it was shown in \cite[Remark~6.4]{DeMaPr09} that we in fact have \begin{equation} \label{00}
  \mon H_\infty(B_{\ell_\infty}) = \mon H_\infty(B_{c_0})\,\,\,
\text{ and }\,\,\,
\mon \mathcal{P}(^m\ell_{\infty})= \mon \mathcal{P}(^{m} c_{0}) \, .
\end{equation}

Let us  collect some more basic facts on sets of monomial convergence which in the following will be used without further reference:
\begin{itemize}
\item
If $z \in \mon H_\infty(B_{\ell_\infty})$,
then every permutation of $z$ is again in $\mon H_\infty(B_{\ell_\infty})$; this was proved in \cite[p.~550]{DeGaMaPG08}.
\item
We know from \cite[p.~29-30]{DeMaPr09} that   $\mon H_\infty(B_{\ell_\infty}) \subset c_0$.
 Hence, the decreasing rearrangement $z^*$ of any $z \in \mon H_\infty(B_{\ell_\infty}) $ is a permutation of $|z|$. This implies that $z \in \mon H_\infty(B_{\ell_\infty})$ if and only if
$z^{*} \in \mon H_\infty(B_{\ell_\infty})$ .
\item
 Let $z \in \mon H_\infty(B_{\ell_\infty})$  and $x = (x_{n})_{n} \in B_{\ell_\infty}$ such that $\vert x_{n} \vert \leq \vert z_{n} \vert$ for
all but finitely many $n$'s. Then $x \in  \mon H_\infty(B_{\ell_\infty})$; this result is from \cite[Lemma~2]{DeGaMaPG08} and was inspired by
\cite[Satz~VI]{Bo13_Goett} (see also Lemma~\ref{annalen}).
\item
Similar results hold  for $\mon \mathcal{P} (^{m} \ell_{\infty})$.
\end{itemize}
 What was so far known on sets of monomial convergence? Bohr \cite{Bo13_Goett} proved
\begin{align}\label{11}
\ell_{2} \cap B_{\ell_\infty}  \subset \mon H_\infty(B_{\ell_\infty})\,,
\end{align}
and Bohnenblust-Hille in \cite{BoHi31}
\begin{align} \label{22}
\ell_{\frac{2m}{m-1}}
 \subset \mon \mathcal{P}(^m\ell_{\infty}).
\end{align}
Moreover,  these two results in a certain sense  are optimal; to see this define
 \begin{equation} \label{binchois}
 \begin{split}
M &:=\sup\big\{1 \leq p \leq \infty \, \big|\, \ell_{p} \cap B_{\ell_{\infty}} \subset \mon H_\infty(B_{\ell_\infty})\big\} \,, \\
M_m & :=\sup\big\{1 \leq p \leq \infty \, \big|\, \ell_{p} \subset \mon \mathcal{P}(^m\ell_{\infty})\big\} \text{ for }  m \in \mathbb{N}\,.
\end{split}
\end{equation}
These are two quantities which measure the size of both sets of convergence in terms of the largest possible
slices $\ell_{p} \cap B_{\ell_{\infty}} $ included in them.
The definition of $M$ (at least implicitly) appears in \cite{Bo13_Goett}, and
\eqref{11} of course gives that $M \geq 2$. The idea of graduating $M$ through $M_m$ appears first in Toeplitz' article \cite{To13}; clearly the estimate
$M_2 \leq 4$ is a reformulation of \eqref{Toeppi}. After Bohr's paper \cite{Bo13_Goett} the intensive search for the exact value of $M$ and $M_m$ was not
succesful for more then 15 years. The final answer was given by Bohnenblust and Hille in \cite{BoHi31}, who were able to prove that
\begin{equation} \label{33}
 M_m = \frac{2m}{m-1}  \,\,\,\,\,\,\, \text{ and } \,\,\,\,\,\,\, M = \frac{1}{2} \,.
\end{equation}
Their original proofs of the upper bounds are clever and ingenious. Using modern techniques of probabilistic nature, different from the original ones, they were improved in
\cite[Example~4.9 and Example~4.6]{DeMaPr09}:
\begin{align} \label{crelle}
\ell_{2} \,\cap\, B_{\ell_\infty}
 \subset \mon H_\infty(B_{\ell_\infty}) \,\subset\,
 \bigcap_{\varepsilon>0} \ell_{2+\varepsilon}\,,
\end{align}
and
\begin{align} \label{gadafi}
\mon \mathcal{P}(^m\ell_{\infty}) \subset \ell_{\frac{2m}{m-1}, \infty}\,.
\end{align}
Recall that for  $1 \leq q < \infty$ the Marcinkiewicz space $\ell_{q, \infty}$  consists of those sequences $z$ for which  $\sup_{n} z^{*}_{n} n^{1/q} < \infty$ (and this supremum defines the norm of this Banach space). Clearly, $\ell_{q, \infty} \subset c_{0}$, hence $z^{*}=(|z_{\sigma(n)}|)$ with $\sigma$ some permutation of $\mathbb{N}$. In Section \ref{prode}
 a simplified proof of \eqref{gadafi} will be given.

\subsection{Statement of the results}

We already mentioned in \eqref{first} that the left inclusion in \eqref{crelle} is strict. The aim of this section is to show that our two sets of monomial convergence can be `squeezed' in a much more drastic way.
Our first theorem gives a complete description of $\mon \mathcal{P}(^m\ell_{\infty})$ and extends all results on this set mentioned so far.

\begin{theorem} \label{polinomios}
Let $m \in \mathbb{N}$. Then
$$\mon \mathcal{P}(^{m} \ell_{\infty}) = \ell_{\frac{2m}{m-1}, \infty}\,\,\,,$$
and moreover there exists a constant $C>0$ such that for every $z \in \ell_{\frac{2m}{m-1}, \infty}$ and every $P \in  \mathcal{P}(^{m} \ell_{\infty})$ we have
\begin{equation}\label{eq:polinomios}
 \sum_{\vert \alpha \vert =m} \vert c_{\alpha}(P) z^{\alpha} \vert
\leq C^m  \Vert z \Vert^{m} \Vert P \Vert \, .
\end{equation}
\end{theorem}
In view of  Bohr's transform $\mathfrak{B}$ from  \eqref{vision} this theorem can be seen as a sort of polynomial counterpart of a recent result on $m$-
homogeneous Dirichlet series. A Dirichlet series $\sum a_{n} n^{-s}$ is called $m$-homogeneous whenever  $a_{n}=0$ for every $\Omega (n) \neq m$;
following standard notation, for each $n \in \mathbb{N}$ we write $\Omega (n)=\vert \alpha \vert$ if $n = p^{\alpha}$ (this counts the prime divisors of
$n$, according to their multiplicity).
 By $\mathcal{H}_{\infty}^{m}$ we denote  the
closed subspace of all $m$-homogeneous Dirichlet series  in the Banach space $\mathcal{H}_{\infty}$.
 Then
the restriction of  the isometric algebra isomorphism $\mathfrak{B}: H_\infty(B_{c_0}) \rightarrow \mathcal{H}_{\infty}$  from \eqref{fund}   defines an isometric and linear bijection:
\begin{equation} \label{fund polin}
  \mathcal{P} (^{m} c_{0}) = \mathcal{H}_{\infty}^{m} \, .
\end{equation}
The following estimate
due to Balasubramanian, Calado and Queff\'{e}lec \cite[Theorem~1.4]{BaCaQu06} is a homogeneous counterpart of
\eqref{oje} and of Theorem \ref{polinomios}: For each $m\geq 1$ there exists  $C_m>0$ such that for every $\sum a_{n} n^{-s} \in \mathcal{H}_{\infty}^{m}$ we have
\begin{equation} \label{ojeje}
 \sum_{n} \vert a_{n} \vert \frac{(\log n)^{\frac{m-1}{2}}}{n^{\frac{m-1}{2m}}}
\leq  C_m \sup_{t \in \mathbb{R}} \Big\vert \sum_{n} a_{n} n^{it} \Big\vert \, \,,
\end{equation}
and  the parameter $\frac{m-1}{2}$  is optimal by \cite[Theorem~3.1]{MaQu10} (here, in contrast to \eqref{eq:polinomios}, it seems unknown whether the constant $C_m$ is subexponential).

At least philosophically holomorphic functions can be viewed as polynomials of degree $m=\infty$. Hence it is not surprising that the complete characterization of $\mon \mathcal{P}(^{m} \ell_{\infty})$ from Theorem \ref{polinomios}  improves  \eqref{11}  and even  the highly non-trivial fact from \eqref{first}: With $$\ell_{2,0} =
\big\{ z \in \ell_{\infty} \,\,\big|\,\, \lim_{n} z^{*}_{n} \sqrt{n} = 0 \big\}$$
we have
\begin{equation} \label{2-0}
\ell_{2} \cap B_{\ell_\infty}   \varsubsetneqq \ell_{2,0} \cap B_{\ell_\infty}\ \subset \mon H_{\infty} (B_{\ell_\infty}) \,;
\end{equation}
note that
by the  prime number theorem
we have  $\big(p_{n}^{-1/2} \big) \in \ell_{2,0} \cap B_{\ell_\infty} $ while this sequence  does not belong to
$\ell_{2}$.
We sketch the proof of \eqref{2-0}: Since $B_{\ell_{2, \infty}}\subset \bigcap_{m\in \mathbb{N}} B_{\ell_{\frac{2m}{m-1}, \infty}}$, by   \eqref{eq:polinomios} and
\cite[Theorem~5.1]{DeMaPr09}  there exists an $r>0$ such that $r B_{\ell_{2, \infty}}
\subset \mon H_{\infty} (B_{\ell_\infty})$. Then we conclude that $\big( \frac{r}{\sqrt{n}} \big)_{n} \in \mon H_{\infty} (B_{\ell_\infty})$ which  easily gives  that $z^{*} \in \mon H_{\infty} (B_{\ell_\infty})$ for every $z \in \ell_{2,0} \cap B_{\ell_\infty}$. By the general remarks on $\mon$ from the beginning of this section this completes the proof.\\

 Improving \eqref{2-0} considerably, the following theorem is our main result on monomial convergence of bounded holomorphic functions on the infinite dimensional polydisk. It can be seen as the power series counterpart of \eqref{oje}, and in Section \ref{D} we will see that it gives far reaching information on the general theory of Dirichlet series.

\begin{theorem} \label{Leonhard}
For each $z \in B_{\ell_\infty}$ the following two statements hold:
\begin{enumerate}
 \item \label{Leo 1} If \, $\displaystyle\limsup_{n \to \infty} \frac{1}{\log n} \sum_{j=1}^{n} z^{* 2}_{j} < 1$, then\,\, $z \in  \mon H_\infty(B_{\ell_\infty})$.
\item \label{Leo 2} If  $z \in  \mon H_\infty(B_{\ell_\infty})$, then\,\,  $\displaystyle\limsup_{n \to \infty} \frac{1}{\log n} \sum_{j=1}^{n} z^{* 2}_{j} \leq 1$\,;\\ moreover, here the converse implication is false.
\end{enumerate}
 \end{theorem}

In the remaining part of this section, we prove Theorems~\ref{polinomios} and \ref{Leonhard}. To do so, we need some  more notation:
Given $k,m \in \mathbb{N}$ we consider the following sets of indices
\begin{align*}
&
 \mathcal{M} (m,k) = \{\mathbf{j} = (j_{1}, \dots , j_{m}) \,\,|\,\, 1 \leq j_{1}, \dots , j_{m} \leq k \} = \{1, \ldots , k \}^{m}
 \\&
 \mathcal{J} (m,k) = \{\mathbf{j} \in \mathcal{M} (m,k)  \,\,|\,\, 1 \leq j_{1} \leq \dots \leq j_{m} \leq k \}
 \\&
\Lambda (m,k) = \{ \alpha \in \mathbb{N}_{0}^{k}  \,\,|\,\, \vert \alpha \vert = \alpha_{1} + \cdots + \alpha_{k} =m \} \, .
\end{align*}
An equivalence relation is defined in $\mathcal{M} (m,k)$ as follows: $\mathbf{i} \sim \mathbf{j}$ if there is a permutation $\sigma$ such that $i_{\sigma (r)} = j_{r}$
for all $r$.
We write $|\mathbf{i}|$ for the cardinality of the equivalence class $[\mathbf{i}]$.
For each $\mathbf{i} \in \mathcal{M} (m,k)$ there is a unique $\mathbf{j} \in \mathcal{J} (m,k)$ such that $\mathbf{i} \sim \mathbf{j}$.
On the other hand, there is a one-to-one relation between  $\mathcal{J} (m,k)$ and $\Lambda (m,k)$: Given $\mathbf{j}$,
one can define $\alpha$ by doing $\alpha_{r} = | \{ q \,|\, j_{q}=r \}|$; conversely, for each $\alpha$, we consider
$\mathbf{j}_{\alpha} = (1, \stackrel{\alpha_{1}}{\dots} , 1, 2,\stackrel{\alpha_{2}}{\dots} ,2 ,
\dots , k ,\stackrel{\alpha_{k}}{\dots} ,k)$. Note that $|\mathbf{j}_{\alpha}| = \frac{m!}{\alpha !}$ for every $\alpha \in \Lambda (m,k)$.
Taking this correspondence into account, the monomial series expansion of a polynomial $P \in \mathcal{P} (^{m} \ell_{\infty}^{k})$ can be expressed in different ways
(we write $c_{\alpha} = c_{\alpha}(P)$)
\begin{equation*} \label{polin varios}
  \sum_{\alpha \in \Lambda(m,k)} c_{\alpha} z^{\alpha} = \sum_{\mathbf{j} \in \mathcal{J}(m,k)} c_{\mathbf{j}} z_{\mathbf{j}}
= \sum_{1 \leq j_{1} \leq \ldots \leq j_{m} \leq k} c_{j_{1} \ldots j_{m}} z_{j_{1}} \cdots z_{j_{m}} \, .
\end{equation*}

\subsection{The probabilistic device} \label{prode}

\noindent The upper inclusions in Theorem \ref{polinomios} and Theorem \ref{Leonhard} are based on the following probabilistic device known as the Kahane-Salem-Zygmund inequality (see e.g. \cite[Chapter 6, Theorem 4]{Ka85}):
There
is a universal constant  $C_{\text{KSZ}} > 0$ such that
for any $m,n$ and any family $(a_{\alpha})_{\alpha \in \Lambda(m,n)}$ of complex numbers there
exists a choice of signs $\varepsilon_{\alpha} = \pm 1$ for which
\begin{equation} \label{Kahane}
\sup_{z \in \mathbb{D}^n} \Big| \sum_{\alpha \in \Lambda(m,n)} \varepsilon_{\alpha} a_{\alpha} z^{\alpha}\Big| \leq C_{\text{KSZ}}\, \sqrt{n \log m \sum_{\alpha} \vert a_{\alpha} \vert^{2}  } \, .
\end{equation}
Let us start with the proof of the upper inclusion of Theorem~\ref{polinomios}. As we have already mentioned  earlier (see \eqref{gadafi}), this result
is from \cite{DeMaPr09}, where it appears as a special case of a more general result proved through more sophisticated probabilistic argument. For the sake of completeness we here prefer to give a direct argument based on \eqref{Kahane}.

\begin{proof}[Proof of the upper inclusion in Theorem~\ref{polinomios}]
Take $z \in \mon \mathcal{P} (^{m} \ell_\infty)$. We show that the decreasing rearrangement $r = z^* \in \ell_{\frac{2m}{m-1}, \infty}$. Since $r \in \mon \mathcal{P} (^{m} \ell_\infty)$, a straightforward closed graph argument (see also \cite[Lemma~4.1]{DeMaPr09}) shows that there is a constant $C(z)> 0$
such that for every $Q \in \mathcal{P} (^{m} \ell_\infty)$ we have
\begin{equation}\label{61}
\sum_{\alpha \in \mathbb{N}_0^{(\mathbb{N})}} \vert c_{\alpha} (Q) r^{\alpha}\vert  \,\,\leq \,\, C(z)\, \Vert Q\Vert \,.
\end{equation}
By \eqref{Kahane} for each $n$
 there are signs $\varepsilon_{\alpha} = \pm 1, \alpha \in  \Lambda(m,n) $ such that the $m$-homogeneous
 polynomial
 $$P(u) = \sum_{\alpha \in \Lambda(m,n)} \varepsilon_{\alpha}  \frac{m!}{\alpha !} u^{\alpha}\,\,, \,\,\, u \in \mathbb{C}^{n}$$
satisfies
\begin{equation} \label{62}
 \Vert P \Vert \leq C_{\text{KSZ}} \sqrt{n \log m \sum_{\alpha \in  \Lambda(m,n)} \big|c_\alpha (P) \big|^{2}} \, .
\end{equation}
But by the multinomial formula we have
 \begin{equation*}
 \sum_{\alpha \in  \Lambda(m,n)} \big|c_\alpha (P) \big|^{2}
=
  \sum_{\alpha \in  \Lambda(m,n)}  \Big( \frac{m!}{\alpha !} \Big)^{2}
\leq m! \sum_{\substack{\alpha \in  \Lambda(m,n)}}  \frac{m!}{\alpha !}  = m! n^{m} \,,
 \end{equation*}
and hence we conclude from  \eqref{61} and \eqref{62} (and another application of the multinomial formula) that
\[
\Bigl(\sum_{j=1}^n r_{j} \Bigr)^m = \sum_{\alpha \in  \Lambda(m,n)}
\frac{m!}{\alpha!}\; r^{\alpha} \leq  C(z)\,C_{\text{KSZ}}  \sqrt{ m!\log m} \,\,\, n^{\frac{m+1}{2}} \,.
\]
Finally, this shows that for all $n$ we have
\[
 r_{n} \leq \frac{1}{n} \sum_{j=1}^n r_{j}   \leq C(z)\, C_{\text{KSZ}}   (m! \log m)^{\frac 1m}\,\, n^{\frac{m+1}{2m}-1}  \ll \frac{1}{n^{\frac{m-1}{2m}}} \, ,
\]
the conclusion.
\end{proof}

A similar argument leads to the

\begin{proof}[Proof of the upper inclusion in Theorem~\ref{Leonhard}]
Let us fix some $z \in \mon H_{\infty}(B_{\ell_\infty})$. Then $z \in B_{c_{0}}$ and without loss of generality we may assume that
$r=z$ is non-increasing. Again a closed graph argument assures that  there is  $C(z)>0$ such that
for every $f \in H_{\infty}(B_{\ell_\infty})$
\begin{equation*} \label{r mon}
 \sum_{ \alpha} \vert c_{\alpha}(f) \vert r^{\alpha} \leq C(z) \big\|f\big\|\,.
\end{equation*}
For each $m,n$ and
$a_\alpha = r^\alpha\,,\,\, \alpha \in \Lambda(m,n)$
we choose signs $\varepsilon_\alpha$ according to \eqref{Kahane}, and define $f(u) = \sum_{\alpha \in  \Lambda (m,n)} \varepsilon_{\alpha} r^{\alpha} u^{\alpha}, \,\, u \in \mathbb{D}^{n} $. Then  the preceding estimate gives
\begin{multline*}
 \sum_{\alpha \in  \Lambda (m,n)} r^{2 \alpha} = \sum_{\alpha \in  \Lambda (m,n)} \vert \varepsilon_{\alpha} r^{\alpha} \vert r^{\alpha}
\leq C(z) \big\| f \big\| \\
\leq C(z) C_{\text{KSZ}} \Big( n \log m \sum_{\alpha \in  \Lambda (m,n)} \vert r^{\alpha} \vert^{2}  \Big)^{\frac{1}{2}}
= A \sqrt{n \log m}\,\,\Big(\sum_{\alpha \in  \Lambda (m,n)} r^{2 \alpha}\Big)^{\frac{1}{2}} \, .
\end{multline*}
This implies
\[
 \Big(\sum_{\alpha \in  \Lambda (m,n)} r^{2 \alpha}\Big)^{\frac{1}{2}}  \leq A \sqrt{n \log m} \,.
\]
Now,
\[
(r_1^2+\dots+r_n^2)^m \leq m!\sum_{\alpha \in  \Lambda (m,n)}r^{2\alpha}.
\]
Using Stirling's formula and taking the power $1/m$, we get
$$r_1^2+\dots+r_n^2\leq A^{\frac{1}{m}}me^{-1}m^{\frac{1}{2m}}n^{\frac{1}{m}}(\log m)^{\frac{1}{m}} \, .$$
We then choose $m=\lfloor \log n\rfloor$ so that $e^{-1}n^{1/m}\to 1$. This yields
$$r_1^2+\dots+r_n^2\leq\log n\times \exp\left(\left(\frac 12+o(1)\right)\frac{\log\log n}{\log n}\right)\,,$$
and we immediately deduce $$\displaystyle\limsup_{n \to \infty} \frac{1}{\log n} \sum_{j=1}^{n} r^{2}_{j} \leq 1\,.$$
Moreover, the converse is false, since if we consider a decreasing sequence $(r_n)$ satisfying, for large values of $n$,
$$r_1^2+\dots+r_n^2=\log n\times\exp\left(\frac{\log\log n}{\log n}\right),$$ then $\displaystyle\limsup_{n \to \infty} \frac{1}{\log n} \sum_{j=1}^{n} r^{ 2}_{j} \leq 1$ whereas $(r_n)\notin \mon H_{\infty}(B_{\ell_\infty})$.
\end{proof}

\begin{remark}
The same argument gives also informations on the constant $C$ appearing in \eqref{eq:polinomios}. More precisely, if there exists $A,C>0$ such that, for every $z\in\ell_{\frac{2m}{m-1},\infty}$ and for every $P\in\mathcal P(^m\ell_\infty)$, we have
\begin{equation}\label{polynomios2}
\sum_{|\alpha|=m}|c_\alpha(P)z^\alpha|\leq AC^m \|z\|^m\|P\|,
\end{equation}
then we claim that $C\geq e^{1/2}$. Indeed, provided \eqref{polynomios2} is satisfied, and arguing as in the proof of Theorem~\ref{Leonhard}, we see that for any $0 <r_1,\dots,r_n$,
$$(r_1^2+\dots+r_n^2)^{m/2}\leq AC^m \|r\|^m C_{\text{KSZ}}\sqrt{m!}\sqrt{n\log m}.$$
We choose $r_j=\frac{1}{j^{\frac{m-1}{2m}}}$ so that $\|r\|=1$ and
$$r_1^2+\dots+r_n^2=\sum_{j=1}^n \frac{1}{j^{1-\frac 1m}}\geq \int_1^n \frac{dx}{x^{1-\frac 1m}}\geq mn^{\frac1m}-m.$$
Hence,
$$C\geq \frac{1}{(AC_{\text{KSZ}})^{\frac 1m}}\times \frac{1}{(\log m)^{\frac{1}{2m}}}\times\frac{1}{(m!)^{\frac{1}{2m}}}\times\left(m-\frac{m}{n^{\frac 1m}}\right)^{\frac 12}.$$
Letting $n$ to infinity and then $m$ to infinity, and using
$$\lim_{m\to+\infty}\frac{m}{(m!)^\frac1m}=e,$$
we get the claim. We will see later that \eqref{polynomios2} is satisfied with $C$ any constant greater than $(2e)^{1/2}$.
\end{remark}

\subsection{Tools}
The proof of Theorem \ref{polinomios} and Theorem~\ref{Leonhard}--(\ref{Leo 1}) share some similarities. They need several lemmas. The first one is a Khinchine-Steinhaus type inequality for $m$-homogeneous polynomials on the $n$-dimensional torus $\mathbb{T}^n$ (see \cite{Ba02} and also \cite{We80}).  Following \cite{Ru69} and \cite{Wo91}
$m_{n}$ and $m$ will denote the product of the normalized Lebesgue measure
respectively on $\mathbb{T}^{n}$ and  $\mathbb{T}^{\infty}$ (i.e. the unique rotation invariant Haar measures).

\begin{lemma} \label{Ste-Kin}
Let $1 \leq r \le s < \infty\,.$ Then for every $m$-homogeneous polynomial $P \in \mathcal{P}(^m \mathbb{C}^n)$
we have
\[
\Big(\int_{\mathbb{T}^n} \big|P(w)  \big|^s dm_{n}(w)\Big)^{1/s}
\,\,
\leq
\,\,
\sqrt{\frac{s}{r}}^m \,\,\Big(\int_{\mathbb{T}^n} \big|P(w)  \big|^r dm_{n}(w)\Big)^{1/r}\,.
\]
\end{lemma}

\noindent The second lemma needed for the proof of Lemma \ref{Fred2} is the following hypercontractive Bohnenblust-Hille inequality for $m$-homogeneous polynomials on the $n$-dimensional torus. This was recently shown in \cite{BaPeSe13},  improving a result from \cite{DeFrOrOuSe11}.

  \begin{lemma} \label{BH-in}
For every $\kappa>1$ there is a constant $C(\kappa) >0$ such that for every $m$-homogeneous polynomial $P = \sum_{|\alpha|=m} c_\alpha z^\alpha\,,\,\, z \in \mathbb{C}^n$
we have
\[
\left(  \sum_{\alpha=\Lambda(m,n)} | c_\alpha   |^{\frac{2m}{m+1}}\right)^{\frac{m+1}{2m}} \,\, \leq \,\, C(\kappa)\,\kappa^m \Vert P  \big\Vert\,.
\]
\end{lemma}

We are now ready to give the main technical tool.

\begin{lemma} \label{Fred2}
Let $n\geq 1$, let $m\geq p\geq 1$ and  let $\kappa>1$. There exists $C(\kappa)>0$ such that,
for any $P\in\mathcal P(^{m} \ell_{\infty}^n)$  with coefficients $(c_{\bj})_{\bj}$, we have
\[
\left[\sum_{\bj\in\jpn}\left(\sum_{\substack{\bi\in\jbn{m-p}\\ i_{m-p}
\leq j_1}}\vert c_{(\bi,\bj)}\vert^2\right)^{\frac12\times\frac{2p}{p+1}} \right]^{\frac{p+1}{2p}}
\leq C(\kappa)\left[\kappa\left(1+\frac 1p\right)\right]^m\Vert P  \Vert \, .
\]
\end{lemma}

\begin{proof}
Let us start by denoting
\[
H:=\left[\sum_{\bj\in\jpn}\left(\sum_{\substack{\bi\in\jbn{m-p}\\ i_{m-p}\leq j_1}}|c_{(\bi,\bj)}|^2\right)^{\frac12\times\frac{2p}{p+1}}\right]^{\frac{p+1}{2p}} \, .
\]
Let $L$ be
the symmetric $m$-linear form associated to $P$, whose coefficients $a_{i_1,\dots,i_m}=L(e_{i_1},\dots,e_{i_m})$ satisfy, for $\bi\in\jmn$,
\[
a_{\bi}=\frac{c_{\bi}}{|\bi|}\,.
\]
We fix $\bj\in\jpn$, and note that for any $\bi\in\jbn{m-p}$
\[
|(\bi,\bj)| \leq m (m-1) \cdots (m-p+1) |\bi|\,.
\]
Then
\begin{multline*}
\sum_{ \substack{ \bi\in\jbn{m-p} \\ i_{m-p}\leq j_1}}|c_{(\bi,\bj)}|^2 \leq \sum_{\bi\in\jbn{m-p}}|(\bi,\bj)|^2 |a_{(\bi,\bj)}|^2
\leq m^{2p}\sum_{\bi\in\jbn{m-p}}|\bi|^2 |a_{(\bi,\bj)}|^2 \, .
 \end{multline*}
We now apply Lemma~\ref{Ste-Kin} with the exponent $\frac{2p}{p+1}$ to the $(m-p)$-homogeneous polynomial $z \mapsto L(z,\dots,z,e_{j_1},\dots,e_{j_p})$ to get
\begin{align*}
\sum_{ \substack{ \bi\in\jbn{m-p} \\ i_{m-p}\leq j_1}}|c_{(\bi,\bj)}|^2
\leq & m^{2p} \left(1+\frac 1p\right)^{m}\\
&\times\left(
\int_{\mathbb{T}^n}\bigg|\sum_{\bi\in\jbn{m-p}} |\bi| a_{(\bi,\bj)} w_{i_1}\cdots w_{i_{m-p}} \bigg|^{\frac{2p}{p+1}} dm_{n} (w) \right)^{\frac{p+1}{2p}\times 2} \, .
\end{align*}
We then sum over $\bj\in\jpn$. This yields
$$
H^{ \frac{2p}{p+1}} \leq m^{\frac{(2p)^2}{2(p+1)}} \left(1+\frac{1}{p} \right) ^{ m\times \frac{2p}{p+1}}
\times \int_{\mathbb{T}^n} \sum_{\bj\in\jpn} \big\vert L(w,\dots,w,e_{j_1}, \dots,e_{j_p} ) \big\vert^{\frac{2p}{p+1}} dm_{n}(w) \, .
$$
For each fixed $w \in \mathbb{T}^{n}$ we apply Lemma~\ref{BH-in} with $1<\kappa_0<\kappa$ to the $p$-homogeneous polynomial $z\mapsto L(w,\dots,w,z,\dots,z)$:
\begin{align*}
H^{ \frac{2p}{p+1}} & \leq C(\kappa_0)m^{\frac{(2p)^2}{2(p+1)}} \left[ \left(1+\frac{1}p \right) \kappa_0 \right]^{m\times\frac{2p}{p+1}}
\sup_{w,z\in \mathbb{T}^n} |L(w,\dots,w,z,\dots,z)|^{\frac{2p}{p+1}}\\
&\leq  C'(\kappa_0)m^{\frac{(2p)^2}{2(p+1)}} \left[  \left(1+\frac{1}{p} \right)\kappa_0 \right]^{m \times \frac{2p}{p+1}} m^{\frac p{p+1}} \|P\|^{\frac{2p}{p+1}}\,,
\end{align*}
where in the last estimate we have used an inequality from Harris \cite[Theorem~1]{Ha75}.

\end{proof}

\subsection{Proof of Theorem \ref{polinomios}--lower inclusion}

Let $z\in\ell_{\frac{2m}{m-1},\infty}$, so that $\sup_n z_n^*n^{\frac{m-1}{2m}}=\|z\|<\infty$.
Let us fix $n\geq 1$ and let us consider $P\in \mathcal P(^m \ell_\infty^n)$ with coefficients $(c_{\bj})_{\bj}$.
Using the Cauchy-Schwarz inequality, we may write
\begin{eqnarray*}
\sum_{\bj\in\jmn}|c_\bj|z_{\bj}^*&\leq&\sum_{j\geq 1}z_j^*\sum_{i_1\leq \dots\leq i_{m-1}\leq j}|c_{(j,\bi)}|
z_{i_1}^*\cdots z_{i_{m-1}}^*\\
&\leq&\sum_{j\geq 1} z_j^* \left(\sum_{i_1\leq \dots\leq i_{m-1}\leq j}|c_{(j,\bi)}| ^2\right)^{1/2}
\left(\sum_{i_1\leq \dots\leq i_{m-1}\leq j}z_{i_1}^{*2}\cdots z_{i_{m-1}}^{*2}\right)^{1/2}.
\end{eqnarray*}
Now,
\begin{eqnarray*}
\sum_{i_1\leq \dots\leq i_{m-1}\leq j}z_{i_1}^{*2}\cdots z_{i_{m-1}}^{*2}&\leq&
\sum_{i_1\leq \dots\leq i_{m-1}\leq j}\frac{\|z\|^{2(m-1)}}{i_1^{\frac{m-1}m}\cdots i_{m-1}^{\frac{m-1}m}}.
\end{eqnarray*}
For $k\leq m$ and $u\leq v$, we have
$$\sum_{u\leq v}\frac{1}{u^{1-\frac km}}\leq\int_0^v \frac{1}{u^{1-\frac km}}du=\frac{m}{k}v^{k/m}.$$
By applying the above inequality for $k=1,\ldots, m-1$,  an easy induction yields
\begin{eqnarray*}
\sum_{i_1\leq \dots\leq i_{m-1}\leq j}z_{i_1}^{*2}\cdots z_{i_{m-1}}^{*2}
&\leq&\sum_{i_{m-1}=1}^j\sum_{i_{m-2}=1}^{i_{m-1}} \ldots \sum_{i_1=1}^{i_2} \frac{\|z\|^{2(m-1)}}{i_1^{1-\frac{1}{m}}i_2^{1-\frac{1}{m}}\cdots i_{m-1}^{1-\frac{1}{m}}}\\
&\leq &\sum_{i_{m-1}=1}^j\sum_{i_{m-2}=1}^{i_{m-1}} \ldots \sum_{i_2=1}^{i_3} \frac{\|z\|^{2(m-1)}m}{i_2^{1-\frac{2}{m}}i_3^{1-\frac{1}{m}}\cdots i_{m-1}^{1-\frac{1}{m}}}\\
&\leq& \ldots\\
&\leq&\|z\|^{2(m-1)}
\frac{m^{m-1}}{(m-1)!}j^{\frac{m-1}m}\\
&\leq& e^{m-1}\|z\|^{2(m-1)} j^{\frac{m-1}m}.
\end{eqnarray*}
We then deduce that
\begin{eqnarray*}
\sum_{\bj\in\jmn}|c_\bj|z^{*}_\bj&\leq&e^{\frac{m-1}{2}}\|z\|^{m}\sum_{j\geq 1}
\left(\sum_{i_1\leq \dots\leq i_{m-1}\leq j}|c_{(j,\bi)}| ^2\right)^{1/2}\leq C^m \|z\|^m \|P\|
\end{eqnarray*}
where the conclusion comes from Lemma \ref{Fred2} with $p=1$. This shows that $z^{\ast} \in \mon \mathcal{P}(^{m} \ell_{\infty})$, and hence the conclusion follows by the general properties of sets of monomial convergence (given at the beginning of this section).

\subsection{Proof of Theorem~\ref{Leonhard}--lower inclusion}
The proof of Theorem~\ref{Leonhard}--(\ref{Leo 1}) is technically more demanding and needs further lemmas.

\begin{lemma} \label{Fred1}
Let $n\geq 1$, $p > 1$ and $\rho>0$, and take $0<r_i<\rho$ for $i=1,\dots,n$.
Then for any sequence $(c_{\bi})_{\bi\in\bigcup_{m\geq p}\jmn}$ of nonnegative real numbers we have
\begin{multline*}
\sum_{m=p}^{\infty}\sum_{\bi\in\jmn}c_\bi r_{i_1}\dots r_{i_m}\leq
\left[\sum_{\bj\in\jpn} \left(r_{j_1}\cdots r_{j_p}\left(\prod_{l=1}^{j_1}\frac 1{1-\left(\frac{r_l}{\rho}\right)^2}\right)^{1/2}\right)^{\frac{2p}{p-1}}\right]^{\frac{p-1}{2p}}\\
\times \left[\sum_{\bj\in\jpn}\left(\sum_{m\geq p}\sum_{\substack{\bi\in\jbn{m-p}\\ i_{m-p}\leq j_1}}\rho^{2(m-p)}c_{(\bi,\bj)}^2\right)^{\frac12\times\frac{2p}{p+1}}\right]^{\frac{p+1}{2p}} \,.
\end{multline*}
\end{lemma}
\begin{proof}
We begin by writing
\begin{multline*}
\sum_{m=p}^{\infty}\sum_{\bi\in\jpn}c_\bi r_{i_1}\dots r_{i_m} \\
=\sum_{\bj\in\jpn}r_{j_1}\cdots r_{j_p}
\sum_{m\geq p}\sum_{\substack{ \bi\in\jbn{m-p}\\i_{m-p}\leq j_1}}\rho^{(m-p)}c_{(\bi,\bj)}\rho^{-(m-p)}r_{i_1}\cdots r_{i_{m-p}} \,.
\end{multline*}
We apply the Cauchy-Schwarz inequality (inside) to get:
\begin{align*}
\sum_{m=p}^{\infty}\sum_{\bi\in\jmn}c_\bi r_{i_1}\dots r_{i_m}\leq&
\sum_{\bj\in \jpn}r_{j_1}\cdots r_{j_p}\left(\sum_{m\geq p}\sum_{\substack{\bi\in\jbn{m-p}\\i_{m-p}\leq j_1}}\rho^{2(m-p)}c_{(\bi,\bj)}^2\right)^{1/2}\\
& \times \left(\sum_{m\geq p}\sum_{\substack{\bi\in\jbn{m-p}\\i_{m-p}\leq j_1}}r_{i_1}^2\cdots r_{i_{m-p}}^2\rho^{-2(m-p)}\right)^{1/2}\\
\leq&\sum_{\bj\in \jpn}r_{j_1}\cdots r_{j_p}\left(\prod_{l=1}^{j_1}\frac1{1-\left(\frac{r_l}{\rho}\right)^2}\right)^{1/2}\\
&\times\left(\sum_{m\geq p}\sum_{\substack{\bi\in\jbn{m-p}\\i_{m-p}\leq j_1}}\rho^{2(m-p)}c_{(\bi,\bj)}^2\right)^{1/2} \,.
\end{align*}
We conclude by applying H\"older's inequality with the couple of conjugate exponents $\frac{2p}{p+1},\ \frac{2p}{p-1}$.
\end{proof}

\noindent  The strategy now will be to bound each factor in the preceding
lemma. The first factor will be controlled by the condition given in Theorem \ref{Leonhard}.
\begin{lemma} \label{Fred3}
Fix $p > 1$, $0<\alpha<\rho$,  and let  $(r_n)_{n\in\mathbb N}$ be a nonincreasing sequence of nonnegative real numbers satisfying, for all $n\geq 1$,
\[
\left\{
\begin{array}{rcl}
r_n&<&\rho\\
\frac 1{\log (n+1)}(r_1^2+\dots+r_n^2) & \leq & \alpha^2 \, .
\end{array}\right.
\]
Then the sequence
\[
\left(\sum_{\bj\in\jpn} \left(r_{j_1}\cdots r_{j_p}\left(\prod_{l=1}^{j_1}\frac 1{1-\left(\frac{r_l}{\rho}\right)^2}\right)^{1/2}\right)^{\frac{2p}{p-1}}\right)_n
\]
is bounded.
\end{lemma}
\begin{proof}
It is enough to prove that
\[
H:=\sum_{j_1=1}^{\infty}r_{j_1}^{\frac{2p}{p-1}}\left(\prod_{l=1}^{j_1}\frac 1{1-\left(\frac{r_l}{\rho}\right)^2}\right)^{\frac{p}{p-1}}\sum_{j_1\leq j_2\leq \dots\leq j_p}(r_{j_2}\cdots r_{j_p})^\frac{2p}{p-1}
\]
is finite.
We first consider the last sum. Because $(r_n)_{n}$ is nonincreasing, it is plain that, for any $n\geq 1$, $r_n\leq \alpha\sqrt{\log (n+1)}/\sqrt{n}$. We will use that there is a constant $A_p \ge 1$ such for all $a \in \mathbb{N}$ we have
\[
\sum_{k \ge a} \frac{(\log (k+1))^{\frac{p}{p-1}}}{k^\frac{p}{p-1} }
\leq A_p   \frac{1 + (\log a)^{\frac{p}{p-1}}}{a^{\frac{1}{p-1}}} \,.
\]
This implies
\begin{multline*}
\sum_{j_1\leq j_2 \leq\cdots \leq j_p}(r_{j_2}\cdots r_{j_p})^{\frac{2p}{p-1}}
\leq
\sum_{j_2,\cdots, j_p\geq j_1} (r_{j_2}\cdots r_{j_p})^{\frac{2p}{p-1}} \\
\ll \left(\sum_{k=j_1}^\infty  \frac{(\log (k+1))^{\frac{p}{p-1}}}{k^\frac{p}{p-1} }\right)^{p-1}
\ll \frac{\big(1 + (\log j_1)^{\frac{p}{p-1}}\big)^{p-1}}{j_1}
\ll \frac{1+(\log j_1)^{p}}{j_1}\,,
\end{multline*}
where the constant in the last inequality only depends on $\alpha$ and $p$. Furthermore,
\[ \prod_{l=1}^{j_1}\frac 1{1-\left(\frac{r_l}{\rho}\right)^2}=\exp\left(-\sum_{l=1}^{j_1}\log\left(1-\frac{r_l^2}{\rho^2}\right)\right) \, .
\]
Let $\varepsilon>0$ be such that $\alpha^2(1+\varepsilon)<\rho^2$. Since $(r_n)_{n}$ goes to zero, there exists some $A>0$ such that
\[
-\sum_{l=1}^{j_1}\log\left(1-\frac{r_l^2}{\rho^2}\right)\leq A+(1+\varepsilon)\sum_{l=1}^{j_1}\frac{r_l^2}{\rho^2}\leq A+(1+\varepsilon)\frac{\alpha^2}{\rho^2}\log j_1
\]
for any $j_1\geq 1$ (use again that $\lim_{x \rightarrow 0} \frac{-\log (1-x)}{x}=1$). This yields
\[
\prod_{l=1}^{j_1}\frac 1{1-\left(\frac{r_l}{\rho}\right)^2}\ll j_1^{\delta}
\]
for some $\delta<1$. Hence,
\[
H \ll\sum_{j_1\geq 1}\frac{(\log (j_1+1))^{\frac p{p-1}}(1+\log j_1)^{p}}{(j_1)^{1+(1-\delta)\frac{p}{p-1}}} \,.
\]
The  last sum is convergent and this completes the proof.
\end{proof}

Finally, we are ready to give the
\begin{proof}[Proof of Theorem~\ref{Leonhard}--(\ref{Leo 1})] Take $z \in B_{\ell_\infty}$ such that
$$ A:=\displaystyle\limsup_{n \to \infty} \frac{1}{\log n} \sum_{j=1}^{n} z^{* 2}_{j} < 1\,.$$
 We write for simplicity $r_n$ for $z_n^{*}$, and we are going to show that $r \in \mon H_{\infty}(B_{\ell_\infty})$ (see the preliminaries). Choose $A < \alpha < \rho <1$.  Moreover, we know that changing a finite number of terms does not change the property $r \in\mon H_{\infty}(B_{\ell_\infty})$ (see again \cite[Lemma~2]{DeGaMaPG08}), hence we may assume that for all $n\geq 1$
 $$\left\{
\begin{array}{rcl}
r_n&<&\rho\\
\frac 1{\log (n+1)}(r_1^2+\dots+r_n^2)&\leq&\alpha^2.
\end{array}\right.$$
Now we choose $p>1$ and $\kappa>1$ such that $\kappa \rho\left(1+\frac{1}{p}\right)<1$, and consider for each fixed $f \in H_\infty(B_{\ell_\infty})$ and for each  $n$ the decomposition
\[
\sum_{\alpha \in \mathbb{N}^n} |c_\alpha| r^\alpha
=
\sum_{m=1}^{p-1}\,\,\sum_{\bj\in\jmn}|c_\bj| r_{j_1}\dots r_{j_m} + \sum_{m=p}^{\infty}\,\,\sum_{\bj\in\jmn}|c_\bj| r_{j_1}\dots r_{j_m}
\]
Since
\[
r \in  \,\,\,  \ell_{ \frac{2k}{k-1},\infty } \,\,\,\,\text{ for all  } \,\,\,  k \,,
\]
we deduce from  Theorem~\ref{polinomios} (here in fact only the weaker version from \eqref{33} is needed) that the first summand is bounded by a constant independent of $n$. Moreover,
by Lemmas \ref{Fred1} and  \ref{Fred3}, the second summand can be  majorized as follows:
$$
 \sum_{m=p}^{\infty}\,\,\sum_{\bj\in\jmn}|c_\bj| r_{j_1}\dots r_{j_m}\ll\left[\sum_{\bj\in\jpn}\left(\sum_{m\geq p}\sum_{\substack{\bi\in\jbn{m-p}\\ i_{m-p}\leq j_1}}\rho^{2m}|c_{(\bi,\bj)}|^2\right)^{\frac12\times\frac{2p}{p+1}}\right]^{\frac{p+1}{2p}}.
 $$
We then apply Minkowki's inequality and Lemma \ref{Fred2}. Using the Taylor series expansion $f=\sum_{m\geq 0}P_m$ we get
\begin{eqnarray*}
 \sum_{m=p}^{\infty}\,\,\sum_{\bj\in\jmn}|c_\bj| r_{j_1}\dots r_{j_m}&\ll&\rho^m\left[\sum_{\bj\in\jpn}\left(\sum_{\substack{\bi\in\jbn{m-p}\\ i_{m-p}\leq j_1}}|c_{(\bi,\bj)}|^2\right)^{\frac12\times\frac{2p}{p+1}}\right]^{\frac{p+1}{2p}}\\
 &\ll&\sum_{m\geq p}\left[\rho\left(1+\frac 1p\right)\kappa\right]^m \|P_m\|.
\end{eqnarray*}
This yields the conclusion, since by the Cauchy inequalities we have that $\|P_m\|\leq \|f\|$.
\end{proof}

\subsection{Dismissing  candidates} \label{sec:comments}
A natural question seems to be whether or not there is a sequence space $X$ (i.e., a vector space $X$ of complex sequences)  such that
$X \cap B_{\ell_\infty} = \mon H_{\infty} (B_{\ell_\infty})$. The first natural candidate to do that job was $\ell_{2}$  (see again \eqref{11}, \eqref{33}, and \eqref{crelle}). But, as we already have seen in \eqref{first}, the sequence $(p_{n}^{-1/2})_{n}$ belongs to  $\mon H_{\infty} (B_{\ell_\infty})$ although it is not in $\ell_{2}$.
The three other natural candidates are the spaces $\ell_{2,0}$, $\ell_{2, \infty}$ and $\ell_{2, \log}$:
\begin{align*}
&
 \ell_{2, 0} = \Big\{ z \in \ell_{\infty} \,\, |\,\,  \lim_n z^{*}_{n} n^{1/2} = 0  \Big\}
 \\&
 \ell_{2, \infty} = \Big\{ z \in \ell_{\infty} \,\, |\,\, \exists c \, \forall n \,:  z^{*}_{n} \leq c \textstyle\frac{1}{\sqrt{n}}  \Big\}
 \\&
 \ell_{2, \log} = \Big\{ z \in \ell_{\infty}\,\, |\,\,\exists c \, \forall n \, :  z^{*}_{n} \leq c \sqrt{\textstyle\frac{\log n}{n}}  \Big\}\,.
\end{align*}
Theorem~\ref{Leonhard} shows that neither $\ell_{2,0}$ nor $\ell_{2, \log}$ are the proper spaces since we have
\begin{equation} \label{candidatos}
  \ell_{2,0}\cap B_{\ell_\infty} \subsetneqq \mathbf{B} \subset \mon H_{\infty} (B_{\ell_\infty})
\subset \mathbf{\bar{B}} \subsetneqq \ell_{2, \log}\cap B_{\ell_\infty}
\end{equation}
(recall the definition of $\mathbf{\bar{B}}$ from \eqref{reformel}).
We prove this: Note first
\begin{equation}  \label{raiz1}
  \Big( \frac{c}{\sqrt{n}} \Big)_{n \in \mathbb{N}}  \,\,\,\,
\begin{cases}
 \in \mon  H_{\infty} (B_{\ell_\infty})   & \text{ for }  c < 1 \\
  \notin \mon  H_{\infty} (B_{\ell_\infty})   & \text{ for } 1 < c
\end{cases}
\end{equation}
since $$\displaystyle\limsup_{n \to \infty} \frac{1}{\log n} \sum_{j=1}^{n} \frac{1}{j} = 1\,.$$
Now, \eqref{raiz1} immediately gives $\ell_{2,0}\cap B_{\ell_\infty} \subsetneqq  \mathbf{B}$.
The last inclusion in \eqref{candidatos} follows from the fact that $\limsup_n \frac{1}{\log n} \sum_1^n z_j^{*2} < \infty$
obviously implies that $z \in \ell_{2, \log}$.
On the other hand,
\[
 \displaystyle\limsup_{n \to \infty} \frac{1}{\log n} \sum_{j=1}^{n} \Big( \sqrt{\frac{\log j}{j}} \Big)^{2}
\geq \limsup_{n \to \infty} \frac{1}{\log n} \sum_{j=1}^{n} \frac{\log 3}{j}
= \log 3 > 1
\]
gives $\Big( \sqrt{\frac{\log n}{n}} \Big)_{n} \not\in \mathbf{\bar{B}}$ and shows that this inclusion is also strict.\qed

\vspace{2mm}

\noindent
In view of \eqref{raiz1}  the following interesting problem  remains open:
 \[\Big( \frac{1}{\sqrt{n}} \Big)_{n} \in \mon  H_{\infty} (B_{\ell_\infty}) \, ?
 \]
In fact, Theorem~\ref{Leonhard}, even proves  that there is no sequence space  $X$
at all for which
$\mon H_{\infty} (B_{\ell_\infty}) = X \cap B_{\ell_\infty}$: Indeed, assume that such an $X$ exists.  By \eqref{raiz1}  we have that
$(\frac{1}{2}n^{-1/2})_{n \ge 9} \in \mon H_{\infty} (B_{\ell_\infty})$, and therefore by assumption
$(\frac{3}{2}n^{-1/2})_{n \ge 9} \in X \cap B_{\ell_\infty}$.
But then, again by assumption, $\big(\frac{3}{2}n^{-1/2}\big)_{n \ge 9} \in  \mon  H_{\infty} (B_{\ell_\infty})$, a
contradiction to \eqref{raiz1}. \qed

\vspace{2mm}

Finally, we compare $\ell_{2,\infty}\cap B_{\ell_\infty}$ with $\mon H_{\infty} (B_{\ell_\infty})$. Again by
 \eqref{raiz1} we see that
$B_{\ell_{2,\infty}}\subset \mon H_{\infty} (B_{\ell_\infty}) \,,$
and moreover that there  are sequences  in $\ell_{2,\infty}\cap B_{\ell_\infty}$ that do not belong to  $\mon H_{\infty} (B_{\ell_\infty})$.
But it also can be shown that
\[
\mon H_{\infty} (B_{\ell_\infty}) \nsubseteq \ell_{2,\infty}\,;
\]
the proof is now slightly more complicated: Take a strictly increasing
sequence of non-negative integers $(n_{k})_{k}$ with $n_{1}>1$,  satisfying that the sequence $\big(\frac{k+1}{n_{k}}\big)_{k}$ is strictly decreasing and
\[
\sum_{k=1}^\infty \frac{k+1}{n_k}<1 \, ;
\]
(take for example $n_k=a^{k^2(k+1)}$ for $a  \in \mathbb{N}$ big enough). Now we define
\[
 r_{j} =
\begin{cases}
 \sqrt{\frac{1}{n_{1}}} & 1\leq j \leq n_{1} \\
 \sqrt{\frac{k+1}{n_{k+1}}} & n_k <j\leq n_{k+1}, \ \ k=1,2,\ldots.
\end{cases}
\]
The sequence $(r_n)$ is  decreasing to 0. Clearly, $n_kr_{n_k}^2=k$ for all $k$. Thus $(r_n)$ does not belong to $\ell_{2,\infty}$. But  for $n>n_{1}$, if $n_k<n\leq n_{k+1}$ and $\lim_k \frac{k+1}{\log n_k}=0$ (a condition satisfied by the above example), then
\begin{multline*}
 \frac{1}{\log n} \sum_{j=1}^{n} r_{j}^{2}=
 \frac{1}{\log n}\big( \sum_{j=1}^{n_{1}} \frac{1}{n_{1}}+\sum_{h=1}^{k-1}\sum_{j=n_h+1}^{n_{h+1}}r_j^2+
\sum_{j=n_k+1}^{n}r_j^2\big)\\
\leq  \frac{1}{\log n}\big( 1+\sum_{h=1}^{k-1}\frac{n_{h+1}-n_h}{n_{h+1}}(h+1)+\frac{n_{k+1}-n_k}{n_{k+1}}(k+1)\big)\\
\leq \frac{1}{\log n_{1}}+\sum_{h=1}^{k-1}\frac{h+1}{\log n_{h+1}}+\frac{k+1}{\log n_k}
< \sum_{h=1}^\infty \frac{h+1}{n_h}+\frac{k+1}{\log n_k}.
\end{multline*}
Hence $\displaystyle\limsup_{n \to \infty} \frac{1}{\log n} \sum_{j=1}^{n} r_j^2 < 1$, and therefore $(r_{n})_{n} \in \mon H_{\infty} (B_{\ell_\infty})$.\qed \\

\section{Series expansion of $\boldsymbol{H_p}$-functions in infinitely  many variables}
We draw now our attention to functions on $\mathbb{T}^{\infty}$, the infinite dimensional polytorus. We recall that $m$ denotes the product of the normalized Lebesgue measure
on $\mathbb{T}^{\infty}$. Given a function   $f  \in L_{p} (\mathbb{T}^{\infty})$, its Fourier coefficients $(\hat{f}(\alpha))_{\alpha \in \mathbb{Z}^{(\mathbb{N})}}$ are defined by
$\hat{f}(\alpha) = \int_{\mathbb{T}^{\infty}} f(w) w^{- \alpha} dm(w)= \langle f, w^{ \alpha} \rangle$ where $w^{\alpha}=w_{1}^{\alpha_{1}}\ldots w_n^{\alpha_n}$ if $\alpha=(\alpha_{1}\ldots \alpha_n, 0,\ldots)$ for $w \in \mathbb{T}^{\infty}$, and  the bracket $\langle \cdot , \cdot \rangle$ refers to the duality between $L_{p} (\mathbb{T}^{\infty})$ and $L_{q} (\mathbb{T}^{\infty})$ for $1/p + 1/q=1$.
With this, for $1 \leq p \leq \infty$ the Hardy spaces are defined as
\begin{equation} \label{guillaumedufay}
H_{p}(\mathbb{T}^{\infty}) = \Big\{ f \in  L_{p} (\mathbb{T}^{\infty}) \,\,\big|\,\, \hat{f}(\alpha) = 0 \,  , \, \,\, \,
\forall \alpha \in \mathbb{Z}^{(\mathbb{N})} \setminus \mathbb{N}_{0}^{(\mathbb{N})} \Big\}\, .
\end{equation}
We will also consider, for each $m$, the following closed subspace
\begin{equation*} \label{Hpm}
 H_{p}^{m}(\mathbb{T}^{\infty}) = \Big\{ f \in H_{p} (\mathbb{T}^{\infty}) \,\,\big|\,\, \hat{f} (\alpha) \neq 0 \, \Rightarrow \, \vert \alpha \vert = m \Big\}
\end{equation*}
of $L_{p} (\mathbb{T}^{\infty})$. By \cite[Section~9]{CoGa86} this is the completion of the $m$-homogeneous trigonometric polynomials (functions on
$\mathbb{T}^{\infty}$ that are finite sums of the form $\sum_{\vert \alpha \vert =m} c_{\alpha} w^{\alpha}$).  It is important to note that
\begin{equation} \label{Hpq}
H_{q}^{m}(\mathbb{T}^{\infty}) =
 H_{p}^{m}(\mathbb{T}^{\infty})\,, \,\, 1 \leq p,q < \infty \,\,\, \text{ and  }\,\,\, m \in \mathbb{N}\,;
\end{equation}
this was first observed in by \cite[9.1 Theorem]{CoGa86} (here it  also follows from Lemma~\ref{Ste-Kin} and a density argument).

In analogy to \eqref{heide} and \eqref{heide1} we define for every for $1 \leq p \leq \infty$ and $m \in \mathbb{N}$ the following two sets of monomial convergence:
\begin{align*} \label{mon Hardy}
&
 \mon H_{p}(\mathbb{T}^{\infty}) = \Big\{ z \in \mathbb{C}^{\mathbb{N}} \,\,\big|\,\, \sum_{\alpha} \vert \hat{f}(\alpha) z^{\alpha} \vert < \infty
\text{ for all }  f \in  H_{p}(\mathbb{T}^{\infty}) \Big\}
\\&
\mon H^m_{p}(\mathbb{T}^{\infty}) = \Big\{ z \in \mathbb{C}^{\mathbb{N}} \,\,\big|\,\, \sum_{\alpha} \vert \hat{f}(\alpha) z^{\alpha} \vert < \infty
\text{ for all }  f \in  H^m_{p}(\mathbb{T}^{\infty}) \Big\}\,.
\end{align*}
Obviously both sets are increasing in $p$.
In sections \ref{The homogeneous case} and \ref{The general case} we will prove that
\begin{equation} \label{MONO}
\mon H_{\infty} (\mathbb{T}^{\infty}) =  \mon H_{\infty} (B_{c_0}) =\mon H_{\infty} (B_{\ell_\infty})
\end{equation}
\begin{equation} \label{MONO1}
\mon H^m_{\infty} (\mathbb{T}^{\infty}) =  \mon \mathcal{P}(^mc_0)=  \mon \mathcal{P}(^m \ell_\infty)\,,
\end{equation}
which by Theorem \ref{polinomios} and Theorem \ref{Leonhard} then in particular implies that
\[
\mon H_{p}(\mathbb{T}^{\infty}) \subset  \mathbf{\overline{B}}\,\,\, \text{ and }\,\,\, \mon H^m_{p}(\mathbb{T}^{\infty}) \subset \ell_{\frac{m-1}{2m},\infty}\,.
\]
But we are going to see in  this section  that a much more precise description is possible.

\subsection{The homogeneous case} \label{The homogeneous case}
The homogeneous case can be solved completely.

\begin{theorem} \label{final}
\[
  \mon H_{p}^{m} (\mathbb{T}^{\infty}) =
\begin{cases}
  \ell_{2}  & \text{ for } 1 \leq p < \infty \\
  \ell_{\frac{2m}{m-1}, \infty} & \text{ for }  p = \infty.
\end{cases}
\]
Moreover, there is  $C>0$ such that if $z \in \mon H_{p}^{m} (\mathbb{T}^{\infty})$ and  $f \in H_{p}^{m} (\mathbb{T}^{\infty})$,
then
\begin{equation} \label{normas final}
\sum_{\vert \alpha \vert =m} \vert \hat{f} (\alpha) z^{\alpha} \vert \leq C^{m} \Vert z \Vert^{m} \Vert f \Vert_{p} \, ,
\end{equation}
where $\Vert z \Vert$ is the norm in the corresponding sequence space (here  $1 \leq C \leq \sqrt{2}$ for $1 \leq p \leq 2$ and $C=1$ for $2 \leq p < \infty$).
\end{theorem}

Again we prepare the proof with some lemmas of independent interest.
 We deal with
two separate situations: $p=\infty$ and $p=2$ (covering the case for arbitrary $1\le p < \infty$). The first case will follow from Theorem~\ref{polinomios}, after showing that
$H_{\infty}^{m}(\mathbb{T}^{\infty})$ can be identified with $\mathcal{P}(^{m} c_{0})$. The basic idea here is, given a polynomial on $c_{0}$, extend
it to $\ell_{\infty}$ and then restrict it to $\mathbb{T}^{\infty}$. Let us very briefly recall how $m$-homogeneous polynomials on a Banach space $X$ can be
extended to its bidual (see \cite[Section~6]{Fl99} or  \cite[Proposition~1.53]{Di99}). First of all, every $m$-linear mapping $A: X \times \cdots \times X \to \mathbb{C}$ has
a unique extension (called the Arens extension) $\tilde{A} : X^{**} \times \cdots \times X^{**} \to \mathbb{C}$ such that for all $j=1, \ldots , n$, all
$x_{k} \in X$ and $z_{k} \in X^{**}$, the mapping that to $z \in X^{**}$ associates $\tilde{A}(x_{1}, \ldots , x_{j-1}, z , z_{j+1}, \ldots , z_{m})$ is weak$^{*}$-continuous.
Now, given $P \in \mathcal{P}(^{m} X)$, we take its associated symmetric $m$-linear form $A$ and define its Aron--Berner extension $\tilde{P} \in  \mathcal{P}(^{m} X^{**})$
by $\tilde{P}(z) = \tilde{A}(z, \ldots , z)$. By \cite[Theorem~3]{DaGa89} we have
\begin{equation} \label{DavieGamelin}
  \sup_{x \in B_{X}} \vert P(x) \vert = \sup_{z \in B_{X^{**}}} \vert \tilde{P}(z) \vert \, .
\end{equation}
Hence, the operator $$AB:\mathcal{P}(^{m} X) \to \mathcal{P}(^{m} X^{**})\,, \,\,\,AB(P) = \tilde{P}$$
is a linear isometry.

\begin{lemma}\label{polco}
The mapping $$\psi : \mathcal{P}(^{m} c_{0})  \to H_{\infty}^{m}(\mathbb{T}^{\infty})\,, \,\,\,\psi (P)(w)= AB(P)(w)$$ is a surjective isometry.
\end{lemma}
\begin{proof}
Let us note first that, by the very definition of the Aron--Berner extension, for each $\alpha \in \mathbb{N}_{0}^{(\mathbb{N})}$, the monomial
$x \in c_{0} \mapsto x^{\alpha}$ is extended to the monomial $z \in \ell_{\infty} \mapsto z^{\alpha}$. Then the set of finite sums of the type
$\sum_{\vert \alpha \vert} c_{\alpha} x^{\alpha}$ is bijectively and isometrically mapped onto the set of $m$-homogeneous trigonometric polynomials. By
\cite[Propositions~1.59 and 2.8]{Di99} the monomials on $c_{0}$ with $\vert \alpha \vert=m$ generate a dense subspace  of $\mathcal{P}(^m c_{0})$. On the other hand,
by \cite[Section~9]{CoGa86} the trigonometric polynomials are dense in $H_{\infty}^{m} (\mathbb{T}^{\infty})$. This gives the result.
\end{proof}

\noindent To deal with the case $1 \leq p < \infty$ we need the following lemma.

\begin{lemma} \label{inclusion}
 $$\mon H_{p}^{m} (\mathbb{T}^{\infty}) \subset \mon H_{p}^{m-1} (\mathbb{T}^{\infty})$$
\end{lemma}
\begin{proof}
 Let $0 \neq z \in \mon H_{p}^{m} (\mathbb{T}^{\infty})$ and $f \in \mon H_{p}^{m-1} (\mathbb{T}^{\infty})$. We choose
$z_{i_{0}} \neq 0$ and define $\tilde{f} (w) = w_{i_{0}} f(w)$. Let us see that $\tilde{f} \in H_{p}^{m} (\mathbb{T}^{\infty})$; indeed, take a sequence
$(f_{n})_{n}$ of $(m-1)$-homogeneous trigonometric polynomials that converges in the space $L_{p}(\mathbb{T}^{\infty})$ to $f$. Each $f_{n}$ is a finite sum of the type
$\sum_{\vert \alpha \vert = m-1} c_{\alpha}^{(n)} w^{\alpha}$. We define for $w \in \mathbb{T}^{\infty}$
\[
 \tilde{f}_{n} (w) = w_{i_{0}} f_{n}(w) = \sum_{\vert \alpha \vert = m-1} c_{\alpha}^{(n)} w_{1}^{\alpha_{1}} \cdots w_{i_{0}}^{\alpha_{i_{0}}+1} \cdots w_{k}^{\alpha_{k}} \, .
\]
Clearly $\tilde{f}_{n}$ is an $m$-homogeneous trigonometric polynomial. Moreover
\begin{multline*}
 \Big( \int_{\mathbb{T}^{\infty}} \vert w_{i_{0}} f_{n}(w) - w_{i_{0}} f(w) \vert^{p}  d m(w) \Big)^{\frac{1}{p}}
= \Big( \int_{\mathbb{T}^{\infty}} \vert w_{i_{0}}\vert^{p}  \vert f_{n}(w) - f(w) \vert^{p}  dm(w) \Big)^{\frac{1}{p}}  \\
\leq \Big( \int_{\mathbb{T}^{\infty}}   \vert f_{n}(w) - f(w) \vert^{p}  d m(w) \Big)^{\frac{1}{p}} \, .
\end{multline*}
The last term converges to $0$, hence $(\tilde{f}_{n})_{n}$ converges in $L_{p}(\mathbb{T}^{\infty})$ to$\tilde{f}$ and $\tilde{f} \in H_{p}^{m} (\mathbb{T}^{\infty})$. We compute
now the Fourier coefficients:
\begin{multline*}
 \hat{\tilde{f}} (\alpha) = \int_{\mathbb{T}^{\infty}} \tilde{f}(w) w^{- \alpha} dm(w)
= \int_{\mathbb{T}^{\infty}} w_{i_{0}} f(w) w^{- \alpha} dm(w) \\
= \int_{\mathbb{T}^{\infty}}  f(w) w_{1}^{-\alpha_{1}} \cdots w_{i_{0}}^{-\alpha_{i_{0}}+1} \cdots w_{n}^{-\alpha_{n}} dm(w)  \\
= \int_{\mathbb{T}^{\infty}}  f(w) w_{1}^{-\alpha_{1}} \cdots w_{i_{0}}^{-(\alpha_{i_{0}}-1)} \cdots w_{n}^{-\alpha_{n}} dm(w)     \\
= \hat{f} (\alpha_{1} , \ldots , \alpha_{i_{0}}-1 , \ldots , \alpha_{n}) \, .
\end{multline*}
That is
\[
  \hat{\tilde{f}} (\alpha) =
\begin{cases}
\hat{f}(\beta) & \text{ if } \alpha = (\beta_{1} , \ldots , \beta_{i_{0}} + 1 , \ldots , \beta_{n}) \\
0  & \text{ otherwise }
\end{cases}
\]
and this gives
\begin{multline*}
 \sum_{\beta} \vert \hat{f} (\beta) z^{\beta} \vert = \frac{1}{\vert z_{i_{0}} \vert} \sum_{\beta} \vert \hat{f} (\beta) z^{\beta} \vert \, \vert z_{i_{0}} \vert \\
= \frac{1}{\vert z_{i_{0}} \vert} \sum_{\beta} \vert \hat{f} (\beta) z_{1}^{\beta_{1}} \cdots z_{i_{0}}^{\beta_{i_{0}}+1} \cdots z_{n}^{\beta_{n}}  \vert
= \frac{1}{\vert z_{i_{0}} \vert} \sum_{\alpha} \vert \hat{\tilde{f}} (\alpha) z^{\alpha} \vert < \infty \, .
\end{multline*}
Hence $z \in \mon H_{p}^{m-1} (\mathbb{T}^{\infty})$.
\end{proof}

\vspace{2mm}

Finally, we are ready to give the

\begin{proof}[Proof of Theorem~\ref{final}]
The case $p=\infty$ follows from Theorem~\ref{polinomios} and Lemma~\ref{polco}. Assume that $1 \le p < \infty$, and  observe that by
\eqref{Hpq} it suffices to handle the case $p=2$. If $z \in \ell_{2}$ and $f \in H_{2}^{m}(\mathbb{T}^{\infty})$, then
we apply the Cauchy-Schwarz inequality and the binomial formula to get
\begin{equation} \label{CauchySchwarz}
\sum_{\substack{\alpha \in \mathbb{N}_{0}^{(\mathbb{N})} \\ \vert \alpha \vert = m}} \vert \hat{f} (\alpha) z^{\alpha} \vert
\leq \Big( \sum_{\substack{\alpha \in \mathbb{N}_{0}^{(\mathbb{N})} \\ \vert \alpha \vert = m}} \vert \hat{f}(\alpha) \vert^{2} \Big)^{\frac{1}{2}}
\Big( \sum_{\substack{\alpha \in \mathbb{N}_{0}^{(\mathbb{N})} \\ \vert \alpha \vert = m}} \vert z \vert^{2 \alpha} \Big)^{\frac{1}{2}}
\leq \Vert z \Vert_{2}^{m} \Vert f \Vert_{2} < \infty \, ;
\end{equation}
this implies $z \in \mon H_{2}^{m}(\mathbb{T}^{\infty})$.
Let us now fix $z \in \mon H_{2}^{1}(\mathbb{T}^{\infty})$. By a closed-graph argument, there is $c_{z}> 0$ such that for every $f \in H_{2}^{1}(\mathbb{T}^{\infty})$
the inequality
$\sum_{n=1}^{\infty} \vert \hat{f}(n) z_{n} \vert \leq c_{z} \Vert f \Vert_{2}$ holds.
We fix  $y \in \ell_{2}$, and define  for each $N$ the function $f_{N} : \mathbb{T}^{\infty} \to \mathbb{C}\,,\,\,f_{N}(w)= \sum_{n=1}^{N} w_{n} y_{n}$. Clearly $\hat{f}(n) = y_{n}$ for $n=1, \ldots , N$. Hence $f \in H_{2}^{1}(\mathbb{T}^{\infty})$, and as a consequence we have
 \[
\sum_{n=1}^{N} \vert \hat{f}(n) z_{n} \vert \leq c_{z} \Big(  \sum_{n=1}^{N} \vert y_{n} \vert^{2} \Big)^{\frac{1}{2}}
\leq  c_{z} \Vert y \Vert_{2} < \infty \, .
\]
This holds for every $N$, hence $\sum_{n=1}^{\infty} \vert \hat{f}(n) z_{n} \vert \leq  c_{z} \Vert y \Vert_{2}$ and, since this holds for every $y \in \ell_{2}$,
we obtain $z \in \ell_{2}$. This gives
\[
 \ell_{2}  \subset \mon H_{2}^{m} (\mathbb{T}^{\infty})
\subset \mon H_{2}^{1} (\mathbb{T}^{\infty}) \subset \ell_{2} \, .
\]
Finally,  \eqref{normas final} follows immediately from Theorem~\ref{polinomios} and Lemma~\ref{polco} for the case $p=\infty$. Moreover, \eqref{CauchySchwarz}
gives  \eqref{normas final} for $2 \leq p < \infty$ with $C = 1$, and  \eqref{CauchySchwarz} combined with Lemma \ref{Ste-Kin} ($s=2$ and $r=1$) give the inequality with
$C \leq \sqrt{2}$ whenever $1 \leq p < 2$.
\end{proof}

\subsection{The general case} \label{The general case}

We address now our main goal of describing $\mon H_{p} (\mathbb{T}^{\infty})$. There
are three significant cases: $p= 1$,  $p=2$, and $p= \infty$. The description of $\mon H_{p}(\mathbb{T}^{\infty})$  for $1 \leq p < \infty$
will follow from the cases $p=1$ and $p=2$, showing that these two coincide.
\begin{theorem} \label{general}$\ $
\begin{enumerate}
  \item \label{general inf} $   \mathbf{B} \,\subset \,\mon H_{\infty}(\mathbb{T}^{\infty}) \,\subset\, \mathbf{\overline{B}}$
  \item \label{general p} $\mon H_{p}(\mathbb{T}^{\infty}) = \ell_{2} \cap  B_{\ell_\infty}$ \,\,for \,\, $1 \leq p < \infty$.
\end{enumerate}
\end{theorem}
\noindent
Again we prepare the proof (which will be given after Lemma \ref{annalen}) by some independently interesting observations.
Part \textit{(\ref{general inf})} is an immediate consequence of Theorem \ref{Leonhard} and the fact that
$$
\mon H_{\infty} (\mathbb{T}^{\infty}) =  \mon H_{\infty} (B_{c_0}) =\mon H_{\infty} (B_{\ell_\infty})\,,
$$
which we already mentioned without proof in \eqref{MONO}: For the second equality  see  \eqref{00} whereas the proof of $\mon H_{\infty} (\mathbb{T}^{\infty}) =  \mon H_{\infty} (B_{c_0})$ is a consequence  of the following theorem due to Cole and Gamelin \cite[11.2 Theorem]{CoGa86} (see also
\cite[Lemma~2.3]{HeLiSe97}). For the sake of completeness we include an elementary direct proof; the statement about the inverse mapping seems to be  new.
\begin{proposition} \label{H intftys}
 There exists a unique surjective isometry $$\phi : H_{\infty} (\mathbb{T}^{\infty}) \rightarrow H_{\infty} (B_{c_0})$$
 such that for every $f \in H_{\infty} (\mathbb{T}^{\infty})$ and every $\alpha \in \mathbb{N}_{0}^{(\mathbb{N})}$ we have
$$c_{\alpha} \big(\phi (f)\big) = \hat{f}(\alpha)\,.$$
Moreover, when restricted to $H_{\infty}^{m} (\mathbb{T}^{\infty})$, the mapping $\psi$ defined in Proposition~\ref{polco} and $\phi$ are inverse to each other.
\end{proposition}
\begin{proof}
First of all, let us note that in the finite dimensional setting the result is true: It is a well known fact (see e.g. \cite[3.4.4 exercise (c)]{Ru69}) that for each $n$ there exists
an isometric bijection $\phi_{n} : H_{\infty} (\mathbb{T}^{n}) \rightarrow H_{\infty} (\mathbb{D}^{n})$ such that
$c_{\alpha} \big(\phi (f)\big) = \tilde{f}(\alpha)$ for every $f \in H_{\infty} (\mathbb{T}^{n})$ and every $\alpha \in \mathbb{N}_{0}^{n}$.
Take now $f \in H_{\infty} (\mathbb{T}^{\infty})$ and fix $n \in \mathbb{N}$; since we can consider  $\mathbb{T}^{\infty} = \mathbb{T}^{n} \times \mathbb{T}^{\infty}$, we write $w  = (w_{1}, \ldots , w_{n} , \tilde{w}_{n})$ $\in \mathbb{T}^{\infty}$. Then we define
$f_{n} : \mathbb{T}^{n}  \rightarrow \mathbb{C}$ by
\[
 f_{n}(w_{1} , \ldots w_{n}) = \int_{\mathbb{T}^{\infty}} f (w_{1}, \ldots , w_{n} , \tilde{w}_{n}) dm(\tilde{w}_{n}) \, .
\]
By the Fubini theorem $f_{n}$ is well defined a.e. and
\[
 \int_{\mathbb{T}^{\infty}} f (w) d m(w)
=   \int_{\mathbb{T}^{n}} \bigg( \int_{\mathbb{T}^{\infty}} f (w_{1}, \ldots , w_{n} , \tilde{w}_{n}) d m (\tilde{w}_{n}) \bigg) d m_{n} (w_{1}, \ldots , w_{n}) \, ,
\]
hence $f_{n} \in L_{\infty} (\mathbb{T}^{n})$. Moreover, for $\alpha \in \mathbb{Z}^{n}$ we have, again by Fubini
\[
 \hat{f}_{n} (\alpha) =
\int_{\mathbb{T}^{n} \times \mathbb{T}^{\infty}} f(w) w^{- \alpha} dm(w) = \hat{f} (\alpha) \, .
\]
Thus $\hat{f}_{n} (\alpha) = \hat{f} (\alpha) = 0$ for every $\alpha \in \mathbb{Z}^{n} \setminus \mathbb{N}^{n}_{0}$ and $f_{n} \in H_{\infty} (\mathbb{T}^{n})$.
Obviously $\Vert f_{n} \Vert_{\infty} \leq \Vert f \Vert_{\infty}$ since the measure is a probability. We take $g_{n} = \phi_{n} (f_{n}) \in H_{\infty}(\mathbb{D}^{n})$. We have
$\Vert g_{n} \Vert_{\infty} = \Vert f_{n} \Vert_{\infty} \leq \Vert f \Vert_{\infty}$ and
\[
 g_{n} (z) = \sum_{\alpha \in \mathbb{N}_{0}^{n}} \hat{f}_{n} (\alpha) z^{\alpha} = \sum_{\alpha \in \mathbb{N}_{0}^{n}} \hat{f} (\alpha) z^{\alpha}
\]
for every $z \in \mathbb{D}^{n}$. Since this holds for every $n$ we can define $g : \mathbb{D}^{(\mathbb{N})} \rightarrow \mathbb{C}$ by
$g(z) = \sum_{\alpha \in \mathbb{N}_{0}^{n}} \hat{f} (\alpha) z^{\alpha}$. We have $\Vert g \Vert_{\infty} = \sup_{n} \Vert g_{n} \Vert_{\infty} \leq \Vert f \Vert_{\infty}$.
 By \cite[Lemma~2.2]{DeGaMa10} there exists a unique extension $\tilde{g} \in H_{\infty} (B_{c_0})$ with $c_{\alpha} (\tilde{g}) =\hat{f} (\alpha)$
and $\Vert \tilde{g} \Vert_{\infty} = \Vert g \Vert_{\infty}  \leq \Vert f \Vert_{\infty}$. Setting $\phi(f) = \tilde{g}$ we have that $\phi : H_{\infty} (\mathbb{T}^{\infty})
\rightarrow H_{\infty} (B_{c_0})$ is well defined and such that  for every $f \in H_{\infty} (\mathbb{T}^{\infty})$ we have $\Vert \phi(f) \Vert_{\infty}
\leq \Vert f \Vert_{\infty}$ and $c_{\alpha} \big(\phi (f)\big) = \tilde{f}(\alpha)$  for every $\alpha \in \mathbb{N}_{0}^{(\mathbb{N})}$.
On the other hand if $f \in L_{\infty}(\mathbb{T}^{\infty})$ is such that $\hat{f}(\alpha) =0$ for all $\alpha$ then $f =0$. Hence $\phi$ is injective.\\
Let us see that it is also surjective and moreover an isometry. Fix $g \in H_{\infty} (B_{c_0})$ and consider $g_{n}$ its restriction to the first $n$ variables.
Clearly $g_{n} \in  H_{\infty} (\mathbb{D}^{n})$ and $\Vert g_{n} \Vert_{\infty} \leq \Vert g \Vert_{\infty}$. Using again \cite[3.4.4 exercise (c)]{Ru69} we can
choose $f_{n} \in   H_{\infty} (\mathbb{T}^{n})$ such that $\Vert f_{n} \Vert_{\infty} = \Vert g_{n} \Vert_{\infty}$ and $c_{\alpha} (g_{n}) = \hat{f}_{n} (\alpha)$
for all $\alpha \in \mathbb{N}_{0}^{n}$. Since $c_{\alpha} (g_{n}) = c_{\alpha} (g)$ we have $\hat{f}_{n} (\alpha)= c_{\alpha} (g)$.
We define now $\tilde{f}_{n} \in H_{\infty} (\mathbb{T}^{\infty})$ by $\tilde{f}_{n} (w) = f_{n} (w_{1} , \ldots , w_{n})$ for
$w \in \mathbb{T}^{\infty}$. Then the sequence $(\tilde{f}_{n})_{n=1}^{\infty}$ is contained in the closed ball in $L_{\infty} (\mathbb{T}^{\infty})$
centered at $0$ and with radius $\Vert g \Vert_{\infty}$. Since this ball is $w^{*}$\!-compact and metrizable,
there is a subsequence $(\tilde{f}_{n_{k}})_{k}$ that $w^{*}$-converges to some $f \in L_{\infty} (\mathbb{T}^{\infty})$
with $\Vert f \Vert_{\infty} \leq \Vert g \Vert_{\infty}$. Moreover,  $\hat{f}(\alpha)
= \langle f , w^{ \alpha} \rangle =$ $ \lim_{k \to \infty} \langle \tilde{f}_{n_{k}} , w^{ \alpha} \rangle = $ $ \lim_{k \to \infty} \hat{\tilde{f}}_{n_{k}} (\alpha)$ for every
$\alpha \in \mathbb{Z}_{0}^{(\mathbb{N})}$ and this implies $f \in H_{\infty} (\mathbb{T}^{\infty})$.
Let us see that  $\phi(f)=g$, which shows that $\phi$ is onto; indeed, if $\alpha = (\alpha_{1}, \ldots , \alpha_{n_{0}}, 0 , \ldots)$ then for $n_{k} \geq n_{0}$ we have
\[
\langle \tilde{f}_{n_{k}} , w^{\alpha} \rangle
= \int_{\mathbb{T}^{\infty}} \tilde{f}_{n_{k}} (w) w^{-\alpha} dm(w)
= \int_{\mathbb{T}^{n_{k}}} f_{n_{k}} (w) w^{-\alpha} dm_{n_{k}}(w)
= \hat{f}_{n_{k}} (\alpha) = c_{\alpha} (g) \, .
\]
Hence  $\hat{f}(\alpha)= c_{\alpha} (g)$ for all $\alpha \in \mathbb{N}_{0}^{(\mathbb{N})}$. Furthermore, since $\Vert f \Vert_{\infty} \leq \Vert g \Vert_{\infty} = \Vert \phi(f) \Vert_{\infty}$ we also get that
$\phi$ is an isometry.
Let us fix $P \in \mathcal{P}(^{m} c_{0})$ and show that $\phi^{-1}(P)(w)= \tilde{P}(w)$ for every $w \in \mathbb{T}^{\infty}$. We
choose $(J_{k})_{k}$ a sequence of finite families of multi-indexes included in $\{\alpha: \alpha\in \mathbb{N}_{0}^{(\mathbb{N})}: \vert \alpha \vert =m\}$ and such that the
sequence  $P_{k}=\sum_{\alpha\in J_k} c_{\alpha,k} x^{\alpha}$ converges uniformly to $P$ on the unit ball of $c_{0}$. Since each $J_k$ is finite, we have
\[
\phi^{-1}(P_k)(w)=\sum_{\alpha\in J_k}c_{\alpha,k}w^{\alpha}=\tilde{P}_k(w) \, ,
\]
for every $w\in \mathbb{T}^\infty$. The linearity of the $AB$ operator and \eqref{DavieGamelin} give that  $\Vert \tilde{P}-\tilde{P}_{k} \Vert= \Vert P-P_{k} \Vert =
\Vert \phi^{-1}(P)-\phi^{-1}(P_{k}) \Vert$ converges to 0 and complete the proof.\\
Observe that this argument actually works to prove that $\phi^{-1}(g)(w)=\tilde{g}(w)$ for every $w\in \mathbb{T}^\infty$ and every function $g$ in the completion of the space
 of all polynomials on $c_{0}$.
\end{proof}

\noindent  We handle now the case $p=2$ of  part \textit{(\ref{general p})}  of Theorem~\ref{general} where slightly more can be said (for the proof of Theorem \ref{general} this will not be needed). Here, since
$H_{2}(\mathbb{T}^{\infty})$ is a Hilbert space with the orthonormal basis  $\{w^{\alpha} \}_{\alpha}$,
 we have $\Vert f \Vert_{2} = \big( \sum_{\alpha} \vert \hat{f}(\alpha) \vert^{2} \big)^{1/2}$ which  simplifies  the problem a lot.
\begin{proposition} \label{H2}
  We have $$\mon H_{2}(\mathbb{T}^{\infty}) = \ell_{2} \cap B_{\ell_\infty}\,,$$
  and for each $z \in \ell_{2} \cap \ell_\infty$ and
$f \in H_{2}(\mathbb{T}^{\infty})$,
\begin{equation} \label{eq:H2}
  \sum_{\alpha \in \mathbb{N}_{0}^{(\mathbb{N})}} \vert \hat{f} (\alpha) z^{\alpha} \vert
\leq \Big( \prod_{n=1}^{\infty} \frac{1}{1-\vert z_{n} \vert^{2}} \Big)^{\frac{1}{2}} \Vert f \Vert_{2} \, .
\end{equation}
Moreover, the constant $\big( \prod_{n} \frac{1}{1-\vert z_{n} \vert^{2}} \big)^{1/2}$ is optimal.
\end{proposition}
\begin{proof}
The fact that $\ell_{2} \cap B_{\ell_\infty} \subset \mon H_{2}(\mathbb{T}^{\infty})$ follows by using the Cauchy-Schwarz inequality in a similar way as in
\eqref{CauchySchwarz}:
\[
  \sum_{\alpha \in \mathbb{N}_{0}^{(\mathbb{N})}} \vert \hat{f} (\alpha) z^{\alpha} \vert
\leq \Big( \sum_{\alpha \in \mathbb{N}_{0}^{(\mathbb{N})}} \vert \hat{f}(\alpha) \vert^{2} \Big)^{\frac{1}{2}}
\Big( \sum_{\alpha \in \mathbb{N}_{0}^{(\mathbb{N})}} \vert z \vert^{2 \alpha} \Big)^{\frac{1}{2}}
= \Vert f \Vert_{2} \Big( \prod_{n=1}^{\infty} \frac{1}{1-\vert z_{n} \vert^{2}} \Big)^{\frac{1}{2}} < \infty \, .
\]
 On the other hand, since $H_{2}^{1}(\mathbb{T}^{\infty}) \subset H_{2}(\mathbb{T}^{\infty})$ we have that $\mon H_{2}(\mathbb{T}^{\infty})$ is a  subset of
$\mon H_{2}^{1}(\mathbb{T}^{\infty})$ and Theorem~\ref{final} gives the conclusion.
To see that the constant in the inequality is optimal, let us fix $z$ in $\mon H_{2}(\mathbb{T}^{\infty})$ and take $c>0$ such that
\[
   \sum_{\alpha \in \mathbb{N}_{0}^{(\mathbb{N})}} \vert \hat{f} (\alpha) z^{\alpha} \vert \leq c \Vert f \Vert_{2} \, .
\]
For each $n \in \mathbb{N}$ we consider the function $f_{n}(w)= \sum_{\alpha \in \mathbb{N}_{0}^{n}} z^{\alpha} w^{\alpha}$ that clearly satisfies
$f_{z} \in H_{2}(\mathbb{T}^{\infty})$ and $\hat{f}_{z}(\alpha) = z^{\alpha}$ for every $\alpha \in \mathbb{N}_{0}^{n}$ (and $0$ otherwise). Hence
\[
  \sum_{\alpha \in \mathbb{N}_{0}^{n}} \vert z^{\alpha} \vert^{2}
= \sum_{\alpha \in \mathbb{N}_{0}^{n}} \vert \hat{f}_{z} (\alpha) z^{\alpha} \vert
\leq c \Vert f \Vert_{2} = c \Big( \sum_{\alpha \in \mathbb{N}_{0}^{n}} \vert z^{\alpha}  \vert^{2} \Big)^{\frac{1}{2}} \, .
\]
This gives
\[
  c \geq  \Big( \sum_{\alpha \in \mathbb{N}_{0}^{n}} \vert z^{\alpha}  \vert^{2} \Big)^{\frac{1}{2}}
= \Big( \prod_{n=1}^{n} \frac{1}{1-\vert z_{n} \vert^{2}} \Big)^{\frac{1}{2}}
\]
for every $n$. Hence $ c \geq  \big( \prod_{n=1}^{\infty} \frac{1}{1-\vert z_{n} \vert^{2}} \big)^{1/2}$ and the proof is completed.
\end{proof}

\vspace{2mm}

In order to extend this result to the general case $1 \le p < \infty$ we need another important lemma -- an $H_{p}$--version of \cite[Satz~VI]{Bo13_Goett} (see also \cite[Lemma~2]{DeGaMaPG08}).
\begin{lemma} \label{annalen}
 Let $z \in \mon H_{p}(\mathbb{T}^{\infty})$  and $x = (x_{n})_{n} \in B_{\ell_\infty}$ such that $\vert x_{n} \vert \leq \vert z_{n} \vert$ for
all but finitely many $n$'s. Then $x \in  \mon H_{p}(\mathbb{T}^{\infty})$.
\end{lemma}
\begin{proof}
 We follow  \cite[Lemma~2]{DeGaMaPG08} and choose $r \in \mathbb{N}$ such that $\vert x_{n} \vert \leq \vert z_{n} \vert$ for all $n > r$.
We also take $a>1$ such that $\vert z_{n} \vert < \frac{1}{a}$ for $n=1 , \ldots , r$. Let $f \in H_{p}(\mathbb{T}^{\infty})$ with $\Vert f \Vert_{p} \leq 1$.
We fix $n_{1}, \ldots , n_{r} \in \mathbb{N}$ and define for each $u \in \mathbb{T}^{\infty}$,
\[
 f_{n_{1} , \ldots , n_{r}}(u) = \int_{\mathbb{T}^{r}} f(w_{1}, \ldots , w_{r}, u_{1}, \ldots ) w_{1}^{-n_{1}} \cdots  w_{r}^{-n_{r}} d m_{r} (w_{1} , \ldots ,w_{r}) \, .
\]
Let us see that $f_{n_{1} , \ldots , n_{r}} \in H_{p}(\mathbb{T}^{\infty})$; indeed, using H\"{o}lder  inequality we have
\begin{align*}
&
 \bigg( \int_{\mathbb{T}^{\infty}}  \vert f_{n_{1} , \ldots , n_{r}}  (u) \vert^{p} dm(u) \bigg)^{\frac{1}{p}}
 \\&
 = \bigg( \int_{\mathbb{T}^{\infty}} \Big\vert \int_{\mathbb{T}^{r}} f(w_{1}, \ldots , w_{r}, u_{1}, \ldots ) w_{1}^{-n_{1}} \cdots  w_{r}^{-n_{r}}
dm_{r} (w_{1} , \ldots , w_{r})  \Big\vert^{p} dm(u) \bigg)^{\frac{1}{p}}
\\&
\leq \bigg( \int_{\mathbb{T}^{\infty}} \Big( \int_{\mathbb{T}^{r}} \vert f(w_{1}, \ldots , w_{r}, u_{1}, \ldots )  \vert^{p}
dm_{r}(w_{1} , \ldots , w_{r})  \Big) dm(u) \bigg)^{\frac{1}{p}} = \Vert f  \Vert_{p} \, .
\end{align*}
Hence $f_{n_{1} , \ldots , n_{r}} \in L_{p} (\mathbb{T}^{\infty})$ and $\Vert f_{n_{1} , \ldots , n_{r}} \Vert_{p} \leq  \Vert f \Vert_{p} \leq 1$.
Now we have, for every multi index $\alpha = (\alpha_{1}, \ldots , \alpha_{k}, 0, \ldots)$
\begin{align*}
 & \hat{f}_{n_{1} , \ldots , n_{r}}(\alpha)
 \\&
=  \int_{\mathbb{T}^{\infty}} f_{n_{1} , \ldots , n_{r}} (u) u^{- \alpha} dm(u)
\\
&
=\bigg( \int_{\mathbb{T}^{\infty}} \int_{\mathbb{T}^{r}} \frac{f(w_{1}, \ldots , w_{r}, u_{1}, \ldots , u_{k} ) }%
{w_{1}^{n_{1}} \cdots  w_{r}^{n_{r}} u_{1}^{ \alpha_{1}} \cdots u_{k}^{ \alpha_{k}}} dm_{r} (w_{1}, \ldots , w_{r}) dm_{k} (u_{1}, \ldots , u_{k}, 0, \ldots) \bigg)  \\&
=  \hat{f}(n_{1} , \ldots , n_{r}, \alpha_{1}, \ldots , \alpha_{k},0, \ldots) \, .
\end{align*}
Therefore
\[
  \hat{f}_{n_{1} , \ldots , n_{r}}(\alpha) =
\begin{cases}
 \hat{f}(n_{1} , \ldots , n_{r}, \alpha_{1}, \ldots , \alpha_{k},0, \ldots)  & \text{ if } \alpha= (0,\stackrel{r}{\ldots} , 0, \alpha_{1}, \ldots
  \alpha_{k},0, \ldots) \\
0 & \text{ otherwise}
\end{cases}
\]
and this implies  $f_{n_{1} , \ldots , n_{r}} \in H_{p} (\mathbb{T}^{\infty})$. Now, using \eqref{Baire}
(below)
and doing exactly the same calculations as in \cite[Lemma 2]{DeGaMaPG08}
we conclude $\sum_{\alpha} \vert \hat{f}(\alpha) x^{\alpha} \vert < \infty$ and  $x$ belongs to $\mon H_{p}(\mathbb{T}^{\infty})$.
\end{proof}

\vspace{2mm} \noindent Finally, we are ready for the

\begin{proof}[Proof of Theorem~\ref{general}--(\ref{general p})]
 Lower inclusion: Let us remark first that
 $$\mon H_{1}(\mathbb{T}^{\infty}) \subset
\mon H_{p}(\mathbb{T}^{\infty})$$
  since $H_{p}(\mathbb{T}^{\infty}) \subset H_{1}(\mathbb{T}^{\infty})$. Then to get the lower bound it is enough to show
that $\ell_{2} \cap  B_{\ell_\infty} \subset \mon H_{1}(\mathbb{T}^{\infty}) $. As a first step we show that
there exists $0<r< 1$ such that $r B_{\ell_{2}} \cap B_{\ell_\infty} \subset \mon H_{1}(\mathbb{T}^{\infty})$.  Let $r < 1 / \sqrt{2}$ and choose $f \in H_{1}(\mathbb{T}^{\infty})$ and $z \in r B_{\ell_{2}} \cap B_{\ell_\infty}$.  Then
$z = r y$ for some $y \in  B_{\ell_{2}}$. By \cite[9.2 Theorem]{CoGa86} there exists a projection
$P_{m} : H_{1}(\mathbb{T}^{\infty}) \to H_{1}^{m}(\mathbb{T}^{\infty})$ such that $\Vert P_{m} g \Vert_{1} \leq \Vert g \Vert_{1}$ for every
$g \in  H_{1}(\mathbb{T}^{\infty})$. We write $f_{m} = P_{m} (f)$ and we have $\hat{f}_{m}(\alpha) = \hat{f}(\alpha)$ if $\vert \alpha \vert = m$ and $0$ otherwise.
Then
\begin{multline*}
 \sum_{\alpha} \vert \hat{f}(\alpha) z^{\alpha} \vert = \sum_{m=0}^{\infty} \sum_{\vert \alpha \vert =m} \vert \hat{f}(\alpha) (r y)^{\alpha} \vert
= \sum_{m=0}^{\infty} \sum_{\vert \alpha \vert =m} \vert \hat{f}_{m}(\alpha) (r y)^{\alpha} \vert \\
\leq  \sum_{m=0}^{\infty} r^{m} (\sqrt{2})^{m} \Vert f_{m} \Vert_{1}
\leq  \sum_{m=0}^{\infty} (r\sqrt{2})^{m} \Vert f \Vert_{1} < \infty \, ,
\end{multline*}
where in the first inequality we used that $y \in \ell_{2}$ and \eqref{normas final}, and in the second  one that the projection
is a contraction.
Take now some $z \in \ell_{2} \cap  B_{\ell_\infty}$. Then $\big( \sum_ {n=n_{0}}^{\infty} \vert z_{n} \vert^{2} \big)^{1/2} < r$ for some $n_{0}$, and we define
$$x=(0, \ldots , 0 , z_{n_{0}}, z_{n_{0}+1}, \ldots ) \in r B_{\ell_{2}}\cap B_{\ell_\infty}\,.$$
As explained
$x \in \mon H_{1}(\mathbb{T}^{\infty})$, and hence  Lemma~\ref{annalen} implies as desired $z \in \mon H_{1} (\mathbb{T}^{\infty})$.
Upper inclusion: Again by \ref{Ste-Kin}  we have $H_{p}^{1}(\mathbb{T}^{\infty})
= H_{2}^{1}(\mathbb{T}^{\infty})$ with equivalent norms. This, together with Theorem~\ref{final}, gives
\[
\mon H_{p} (\mathbb{T}^{\infty}) \subset \mon H_{p}^{1}(\mathbb{T}^{\infty})  = \mon H_{2}^{1}(\mathbb{T}^{\infty})
\subset \ell_{2} \cap B_{\ell_\infty} \, . \qedhere
\]
\end{proof}

\vspace{2mm}

\begin{remark} \label{evaluacion}
Denote by  $\mathcal{P}_{\text{fin}}$  the space of all trigonometric polynomials on $\mathbb{C}^{\mathbb{N}}$ (all finite sums $\sum_{\alpha\in J} c_{\alpha} z^{\alpha}$). For  each $z \in \ell_\infty$ the evaluation mapping
$$\delta_{z} : \mathcal{P}_{\text{fin}} \rightarrow \mathbb{C}\,, \,\,\,\delta_{z}(f)=f(z)$$
is
clearly well defined. One of the main problems considered in \cite{CoGa86} is to determine for which $z$'s the evaluation mapping $\delta_{z}$  extends continuously to the whole space $H_{p}(\mathbb{T}^{\infty}), \, 1 \le p < \infty$.
This can be reformulated as to describe the following set
\[
 \big\{ z \in \ell_\infty \,\,\big|\,\, \exists c_{z}>0\, \forall \,f \in \mathcal{P}_{\text{fin}} \, : \, \, \vert f(z) \vert \leq c_{z} \Vert f \Vert_{p} \big\} \, .
\]
Since for each $f  \in \mathcal{P}_{\text{fin}}$ and every $\alpha$ we have $\hat{f}(\alpha)=c_{\alpha}$, the previous set can be written as
\begin{equation} \label{Ep}
  \Big\{ z \in \ell_\infty\,\,\big|\,\,\exists c_{z}>0\, \forall \,f \in \mathcal{P}_{\text{fin}} \, : \, \,
\Big\vert \sum_{\alpha} \hat{f}(\alpha) z^{\alpha} \Big\vert \leq c_{z} \Vert f \Vert_{p} \Big\} \, .
\end{equation}
In \cite[8.1 Theorem]{CoGa86} it is shown that for $1 \leq p < \infty$ the set in \eqref{Ep} is exactly $\ell_{2} \cap B_{\ell_\infty}$.
By a closed-graph argument, for each $1 \leq p < \infty$ a sequence $z$ belongs to the set $\mon H_{p}(\mathbb{T}^{\infty})$ if and only if there exists $c_{z}>0$ such that for every
$f \in H_{p}(\mathbb{T}^{\infty})$
\begin{equation} \label{Baire}
\sum_{\alpha \in \mathbb{N}_{0}^{(\mathbb{N})}} \vert \hat{f}(\alpha) z^{\alpha} \vert
\leq c_{z} \Big( \int_{\mathbb{T}^{\infty}} \big\vert f( w) \big\vert^{p} dm(w)\Big)^{\frac{1}{p}} \,.
\end{equation}
This implies
\begin{equation} \label{mondom rewritten}
 \mon H_{p} (\mathbb{T}^{\infty}) =   \Big\{ z \in \ell_\infty
 \,\,\big|\,\,\exists c_{z}>0\, \forall \,f \in \mathcal{P}_{\text{fin}} \, :
  \, \,
\sum_{\alpha} \vert  \hat{f}(\alpha) z^{\alpha} \vert \leq c_{z} \Vert f \Vert_{p} \Big\} \, .
\end{equation}
In view of  \eqref{mondom rewritten} we have that $\mon H_{p} (\mathbb{T}^{\infty})$ is contained in the set in  \eqref{Ep}. Then the upper inclusion
in Theorem~\ref{general}-\textit{(\ref{general p})} follows from \cite[8.1 Theorem]{CoGa86}. The proof we presented here is independent from that in \cite{CoGa86}.
But the lower inclusion in Theorem~\ref{general}-\textit{(\ref{general p})} is stronger than the result in \cite{CoGa86}.\\

\end{remark}

\subsection{Representation of Hardy spaces}

We have seen in Proposition~\ref{H intftys} how, like in the finitely dimensional case, the Hardy space $H_{\infty} (\mathbb{T}^{\infty})$ can be represented as a
space of holomorphic functions on $c_0$. In \cite[10.1 Theorem]{CoGa86} it is proved that every element of $H_p(\mathbb{T}^\infty)$ can be represented by an holomorphic function of bounded type on $B_{\ell_\infty} \cap \ell_{2}$. A characterization of the holomorphic functions coming from elements of $H_p(\mathbb{T}^\infty)$  can be given for $1\leq p<\infty$, in terms of the following Banach space $$H_{p}(B_{\ell_\infty}\cap \ell_{2})$$ of all holomorphic functions $g: B_{\ell_\infty} \cap \ell_{2} \rightarrow \mathbb{C}$  (here $B_{\ell_\infty}\cap \ell_{2}$ is considered as a complete  Reinhardt domain in $\ell_2$) for which
$$
\Vert g \Vert_{H_{p}(B_{\ell_\infty}\cap \ell_{2})} =
\sup_{n \in \mathbb{N}} \sup_{0<r<1} \Big( \int_{\mathbb{T}^{n}} \vert g(rw_{1}, \ldots , r w_{n},0,0, \ldots) \vert^{p} dm_{n}(w_{1}, \ldots , w_{n}) \Big)^{\frac{1}{p}} < \infty  .
$$

\begin{theorem} \label{Hps}
For each $1\leq  p<\infty$ the mapping
$$\phi : H_{p}(\mathbb{T}^{\infty}) \rightarrow   H_{p}(B_{\ell_\infty}\cap \ell_{2})$$
defined by
$$\phi(f)(z)= \sum_{\alpha \in \mathbb{N}_{0}^{(\mathbb{N})}} \hat{f}(\alpha) z^{\alpha}\,\,, \,\,\,z \in B_{\ell_\infty} \cap \ell_{2}$$
is an onto isometry.

\end{theorem}
\begin{proof}
Let us begin by noting that for each fixed $n$ the mapping $$\phi_{n} : H_{p}(\mathbb{T}^{n}) \rightarrow H_{p}(\mathbb{D}^{n})\,\,,\,\,\,\phi_{n}(f)(z)= \sum_{\alpha \in \mathbb{N}_{0}^{n}} \hat{f}(\alpha) z^{\alpha} $$
 is an isometric isomorphism, where $H_{p}(\mathbb{D}^{n})$ denotes
the Banach space of all holomorphic functions $g : \mathbb{D}^{n} \rightarrow \mathbb{C}$ such that
\[
  \Vert g \Vert_{H_{p}(\mathbb{D}^{n})}=
\sup_{0<r<1} \Big( \int_{\mathbb{T}^{n}} \vert g(rw_{1}, \ldots , rw_{n}) \vert^{p} dm_{n}(w_{1}, \ldots , w_{n}) \Big)^{\frac{1}{p}} < \infty \, .
\]
We show in first place that $\phi$ is well defined and a contraction.
Fix $f \in H_{p}(\mathbb{T}^{\infty})$; we know from Theorem~\ref{general} that $\sum_{\alpha \in  \in  \mathbb{N}_{0}^{(\mathbb{N})}} \vert \hat{f}(\alpha) z^{\alpha} \vert < \infty$ for every
$z \in B_{\ell_\infty} \cap \ell_{2}$, hence the series defines a G\^{a}teaux-differentiable function on $B_{\ell_\infty} \cap \ell_{2}$. We denote the $m$-th Taylor
polynomial of $\phi(f)$ at $0$ by $P_{m}$. Since for all $z \in \ell_{2}$
\[
  P_{m}(z)= \sum_{\substack{\alpha \in \mathbb{N}_{0}^{n} \\ \vert \alpha \vert =m}}  \hat{f}(\alpha) z^{\alpha}\,,
\]
 we deduce from \eqref{normas final} that $P \in \mathcal{P}(^{m} \ell_{2})$ and hence $\phi (f)$ defines a holomorphic function on
$B_{\ell_\infty} \cap \ell_{2}$ (see e.g. \cite[Example~3.8]{Di99}). Let us see now that it actually belongs to $H_{p}(B_{\ell_\infty}\cap \ell_{2})$.
Following the notation in Proposition~\ref{H intftys} we define for each $n$
\[
  f_{n}(w_{1}, \ldots , w_{n})  = \int_{\mathbb{T}^{\infty}} f(w_{1}, \ldots , w_{n}, \tilde{w}_{n}) dm(\tilde{w}_{n}) \, ,
\]
where $(w_{1}, \ldots , w_{n})\in \mathbb{T}^n$. By Fubini's theorem and since $L_{p}(\mathbb{T}^{\infty}) \subset L_{1}(\mathbb{T}^{\infty})$, this function is well defined. On the other hand,
using H\"{o}lder's inequality and again Fubini's theorem we get
\begin{align*}
  \int_{\mathbb{T}^{n}} & \vert f_{n} (w_{1}, \ldots , w_{n}) \vert ^{p} dm_{n}  (w_{1}, \ldots , w_{n})  \\
& = \int_{\mathbb{T}^{n}} \Big\vert \int_{\mathbb{T}^{\infty}} f(w_{1}, \ldots , w_{n}, \tilde{w}_{n}) dm(\tilde{w}_{n}) \Big\vert^{p} dm_{n}  (w_{1}, \ldots , w_{n})  \\
& \leq \int_{\mathbb{T}^{n}} \Big( \int_{\mathbb{T}^{\infty}} \vert f(w_{1}, \ldots , w_{n}, \tilde{w}_{n}) \vert dm(\tilde{w}_{n}) \Big)^{p} dm_{n}  (w_{1}, \ldots , w_{n})  \\
& \leq \int_{\mathbb{T}^{n}} \Big( \int_{\mathbb{T}^{\infty}} \vert f(w_{1}, \ldots , w_{n}, \tilde{w}_{n}) \vert^{p} dm(\tilde{w}_{n}) \Big) dm_{n}  (w_{1}, \ldots , w_{n})\,,
\end{align*}
and this implies $f \in L_{p} (\mathbb{T}^{n})$ and $\Vert f_{n} \Vert_{p} \leq \Vert f \Vert_{p}$ for all $n$. Moreover, for $\alpha \in \mathbb{Z}^{n}$ we have
(again using Fubini) $\hat{f}_{n}(\alpha) = \hat{f} (\alpha)$ and $f_{n} \in H_{p} (\mathbb{T}^{n})$. Then $\Vert \phi( f_{n}) \Vert_{H_{p}(\mathbb{D}^{n})}
= \Vert  f_{n} \Vert_{p} \leq \Vert f \Vert_{p}$ for all $n$,  and we arrive at
\begin{equation} \label{IV}
\sup_{n \in \mathbb{N}} \sup_{0<r<1} \int_{\mathbb{T}^{n}} \big\vert \sum_{\alpha \in  \mathbb{N}_{0}^{n}}  \hat{f}(\alpha) (rw)^{\alpha} \big\vert
dm_{n}  (w_{1}, \ldots , w_{n}) \leq \Vert f \Vert_{p} < \infty \, .
\end{equation}
Clearly, $\phi(f)(z_{1}, \ldots , z_{n},0 \ldots) = \sum_{\alpha} \hat{f}(\alpha)
z^{\alpha}$ for every $(z_{1}, \ldots , z_{n}) \in  \mathbb{D}^{n}$, and by \eqref{IV} this implies
\[
\phi(f) \in H_{p}(B_{\ell_\infty}\cap \ell_{2})\,\,\,\text{ and  } \,\,\,  \Vert \phi(f) \Vert_{H_{p}(B_{\ell_\infty}\cap \ell_{2}))} \leq \Vert f \Vert_{p} \, .
\]
Finally,  we  show that $\Phi$ is also an isometry onto: Fix some  $g \in H_{p}(B_{\ell_\infty}\cap \ell_{2})$, and denote by $g_{n}$ its restriction to the first $n$ variables. Then, by definition $g_{n} \in H_{p}(\mathbb{D}^{n})$ and $\Vert g_{n} \Vert_{H_{p}(\mathbb{D}^{n})} \leq
\Vert g \Vert_{H_{p}(B_{\ell_\infty}\cap \ell_{2})}$. Let us take
$f_{n} = \phi_{n}^{-1}(g_{n}) \in H_{p} (\mathbb{T}^{n})$ and define
$$\tilde{f}_{n} : \mathbb{T}^{\infty} \to \mathbb{C}\,\,,\,\,\,\tilde{f}_{n}(w) = f_{n}(w_{1}, \ldots , w_{n})\,.$$  Since we can do this
for every $n$, we have a sequence $(\tilde{f}_{n})_{n}$ contained in the closed ball of $L_{p} (\mathbb{T}^{\infty})$ centered in $0$ and with radius $\Vert g \Vert_{H_{p}(B_{\ell_\infty}\cap \ell_{2})}$, that is a weak-$(L_p, L_q)$-compact set  if $1<p<\infty$. Since $L_{q}(\mathbb{T}^{\infty})$ is separable, the weak$^{*}$-topology
is metrizable, and hence there exists a subsequence $(\tilde{f}_{n_{k}})_{k}$ that weak$^{*}$ converges to some $f \in L_{p} (\mathbb{T}^{\infty})$. For each
$\alpha \in \mathbb{Z}^{(\mathbb{N})}$ we then have
\begin{equation}\label{analytic}
  \hat{f}(\alpha) = \langle f , w^{\alpha} \rangle = \lim_{k} \langle \tilde{f}_{n_{k}} , w^{\alpha} \rangle = \hat{\tilde{f}}_{n_{k}}(\alpha) = c_{\alpha} (g) \,.
\end{equation}
Hence $f\in H_{p}(\mathbb{T}^{\infty})$,  $\phi(f)= g$ and, moreover, $\Vert f \Vert_{p} \leq \Vert g \Vert_{H_{p}(B_{\ell_\infty}\cap \ell_{2})}$. This completes the proof for $1<p<\infty$, and it remains to check the case $p=1$:
Using Riesz' representation theorem (for the dual of $C(\mathbb{T}^{\infty})$), we only obtain that there exists a subsequence
$(\tilde{f}_{n_{k}})_{k}$ that weak$^{*}$ converges to some complex measure $\nu$ on $\mathbb{T}^{\infty}$. But as in \eqref{analytic}, we have $\hat{\nu}(\alpha)=c_{\alpha} (g)$ for every $\alpha \in \mathbb{Z}^{(\mathbb{N})}$. In particular, $\hat{\nu}(\alpha)=0$ for every $\alpha \in \mathbb{Z}^{(\mathbb{N})}\setminus \mathbb{N}_0^{(\mathbb{N})}$,  i.e., $\nu$ is an analytic measure. But in \cite{HelLow58} it is proved that any analytic measure on a topological group is absolutely continuous with respect to the corresponding Haar measure that in our case is $dm$. Hence, $\nu$ can be represented by an element $f\in H_1(\mathbb{T}^\infty)$, and we hence l have, exactly as above, $\phi(f)= g$ and $\Vert f \Vert_{1} \leq \Vert g \Vert_{H_{1}(B_{\ell_\infty}\cap \ell_{2})}$.
\end{proof}

We want to thank  Jan-Fredrik Olsen and Eero Saksman who  very recently informed us about the theorem  of Helson-Lowdenslager from  \cite{HelLow58} (the
M. and F. Riesz theorem for topological groups) which in the preceding proof was essential to handle the case $p=1$. Actually they, together with A. Aleman,   in \cite[Corollary~1]{AlOlSa14} give a direct proof of that result for the infinite dimensional polytorus. In \cite[Corollary 3]{AlOlSa14} they apply their result to obtain a variant of Theorem~\ref{Hps}.

\section{$\boldsymbol{\ell_1}$-multipliers of $\boldsymbol{\mathcal{H}_p}$-Dirichlet series} \label{D}

Finally, we come back to one of our original motivations. We  use Bohr's transform from \eqref{vision} to deduce from our main results  on sets of monomial convergence (see  \ref{polinomios}, \ref{Leonhard},   \ref{final}, and \ref{general}) multiplier theorems for spaces of Dirich\-let series.

Historically all  results on
sets of monomial convergence (at least those of \eqref{Toeppi},
\eqref{11},\eqref{22}, and \eqref{33})
were  motivated through the theory of Dirichlet series.
 Maximal domains
where such Dirichlet series $D=\sum_n a_n n^{-s}$ converge conditionally, uniformly or absolutely are half planes $[\re  > \sigma]$, where $\sigma=\sigma_c,\sigma_u$ or $\sigma_a$ are called the abscissa of conditional, uniform or absolute convergence, respectively. More precisely,
$\sigma_\alpha (D)$ is the infimum of all $r\in\mathbb{R}$ such that  on $[\re  >r]$ we have convergence of $D$ of the requested type $\alpha = c,u$ or $a$.
Each Dirichlet series $D$ defines  a holomorphic  function $d: [\re  > \sigma_c] \rightarrow \mathbb{C}$. If $\sigma_b(D)$ denotes the abscissa of boundedness, i.e. the
infimum of all $r\in\mathbb{R}$ such that  $d$ on the half plane $[\re  >r]$ is bounded, then one of the fundamental theorems of Bohr from \cite{Bo13} is
\begin{align}\label{Bo13}
\sigma_u(D) = \sigma_b(D)\, .
\end{align}
Bohr's so called  \textit{absolute convergence problem} from \cite{Bo13_Goett} asked for the largest possible width of the strip in $\mathbb{C}$ on which a Dirichlet series may
converge uniformly but not absolutely. In other terms, Bohr defined the number
\begin{equation} \label{S}
S := \sup_D \sigma_a(D)-\sigma_u(D)\,,
\end{equation}
where
the supremum is taken over all possible Dirichlet series $D$, and asked for its precise value.

Using the prime number theorem Bohr in \cite{Bo13_Goett} proved that $S=\frac{1}{M}\,,$
and  concluded from \eqref{11} that $S \leq 1/2$
(for the  definition of $M$ see again \eqref{binchois}). Shortly after that Toeplitz with his result from \eqref{Toeppi} got   $1/4 \leq S \leq 1/2$.
Although the general theory of Dirichlet series  during the first decades of the last century was one of the most  fashionable   topics in analysis (with Bohr's  absolute
convergence problem   very much in its focus),
the question whether or not $S=1/2$ remained open for a long period.
Finally,  Bohnenblust and Hille \cite{BoHi31} in 1931 in a rather ingenious fashion answered the  problem in the positive. They proved $\eqref{33}$, and got as a
consequence what we now call the Bohr-Bohnenblust-Hille theorem:
\begin{equation} \label{BBHthm}
S=\frac{1}{2}\,.
\end{equation}
Equivalently we see  by  \eqref{Bo13} that
\begin{equation} \label{jetztdoch}
\sup_{D \in \mathcal{H}_{\infty}} \sigma_{a}(D) =\frac{1}{2}\,,
\end{equation}
 i.e.,
for each $\varepsilon >0$ and each series $\sum_n a_n n^{-s} \in \mathcal{H}_\infty$ we have
$\sum_n |a_n| n^{-\frac{1}{2}- \varepsilon} < \infty\,,$
and moreover $\frac{1}{2}$ here can not be improved. A non trivial consequence of  \eqref{oje} is that this  supremum is attained.

One of the crucial ideas in the Bohnenblust-Hille approach is that they graduate Bohr's problem: They (at least implicitly) observed that
$S_m=\frac{1}{M_m}$ \,,
where
\begin{equation} \label{Sm}
S_m=\sup \sigma_a(D) - \sigma_u(D)\,,
\end{equation}
the supremum now taken over all $m$-homogeneous Dirichlet series
(recall the definition of $M_m$ in \eqref{binchois}). This allows to deduce from $\eqref{33}$ the lower bound
\begin{equation} \label{BHthm1}
S_m  = \frac{m-1}{2m} \,,
\end{equation}
and hence since $S_m \le S $ in the limit case  as desired $ \frac{1}{2} \leq  S$.

\subsection{Main results}
We finally introduce our concept of $\ell_1$-multipliers for $\mathcal{H}_p$-Dirichlet series.
For $1 \leq p \leq  \infty$ the  image
of $H_p(\mathbb{T}^\infty)$ under the Bohr transform $\mathfrak{B}$ defined in \eqref{vision} is denoted by $\mathcal{H}_p$.
 Together with the norm
$\|D\|_{\mathcal{H}_p}= \|\mathfrak{B}^{-1}(D)\|_{H_{p}(\mathbb{T}^{\infty})}$
the vector space of all these so called $\mathcal{H}_p$-Dirichlet series $D = \sum_n a_n n^{-s}$
forms a Banach space. In other words by definition  we identify
\[
\mathcal{H}_p = H_{p}(\mathbb{T}^{\infty})\,.
\]
Similarly, we denote by $\mathcal{H}^m_p$ the image of $H^m_p(\mathbb{T}^\infty)$ under $\mathfrak{B}$, a
closed subspace of $\mathcal{H}_p$  (see e.g. \cite{Ba02} and \cite{QQ13}).

Let $\mathcal{E}$ be a set of Dirichlet series (in our setting we typically have $\mathcal{E} = \mathcal{H}_p$ or $\mathcal{H}^m_p$).
A sequence $(b_n)$ of complex numbers  is said
to be an $\ell_1$-multiplier for $\mathcal{E}$ whenever
\[
\sum_{n=1}^\infty |a_nb_n| < \infty
\]
for all $\sum_n a_n n^{-s} \in \mathcal{E}$. Recall that a sequence $(b_n)$ of complex numbers is said to be multiplicative (or completely multiplicative) whenever  $b_{nm} = b_n b_m$ for all $n,m$.

The Bohr mapping \eqref{vision} links the concept of  multiplicative $\ell_1$-multipliers with our  previous concept  of sets of monomial convergence.

\begin{remark} \label{multikulti}
Let $(b_n)$ be a multiplicative sequence of complex numbers, and $1 \le p \le \infty$. Then $(b_n)$ is an $\ell_1$-multiplier for  $\mathcal{H}_p$  if and only if
$(b_{p_k}) \in \mon H_p(\mathbb{T}^\infty)$. Clearly, an analogous equivalence holds whenever  we replace
 $\mathcal{H}_p$ by $\mathcal{H}^m_p$.
\end{remark}

We now give an almost complete  characterization of multiplicative $\ell_1$-multi\-pliers of $\mathcal{H}_p$-Dirichlet series.
The following theorem can be considered as the highlight of this article since it in a very condensed way contains almost all the information given.
 Recall again that for each bounded  sequence $z=(z_n)$ of complex numbers we define
\[
\boldsymbol{b} (z)=\Big( \limsup_{n \rightarrow \infty}  \frac{1}{\log n} \sum_{j=1}^{n} z^{* 2}_{j}\Big)^{1/2}\,.
\]
\begin{theorem} \label{multiplier}
Let $(b_n)$ be a multiplicative sequence of complex numbers,  $1 \leq p < \infty$ and $m \in \mathbb{N}$.
\begin{enumerate}
\item
\begin{enumerate}
\item  \label{m1i}
$(b_n)$ is an $\ell_1$-multiplier for  $\mathcal{H}^m_p$ if and only if $(b_{p_k}) \in \ell_2$\,.
\item  \label{m1ii} $(b_n)$ is an
$\ell_1$-multiplier for  $\mathcal{H}^m_\infty$  if and only if $(b_{p_j}) \in \ell_{\frac{2m}{m-1}, \infty}$\,.
\end{enumerate}
\item
\begin{enumerate}
\item  \label{m2i} $(b_n)$ is an
$\ell_1$-multiplier for  $\mathcal{H}_p$    if and only if $|b_{p_j}| < 1$ for all $j$ and $(b_{p_k}) \in \ell_2$\,.
\item \label{m2ii}
$(b_n)$ is an
$\ell_1$-multiplier for $\mathcal{H}_\infty$ provided we have that $|b_{p_j}| < 1$ for all $j$ and
$\boldsymbol{b}\big((b_{p_j})\big) < 1$\,.\\
Conversely, if $(b_n)$ is $\ell_1$-multiplier  for $\mathcal{H}_\infty$,  then $|b_{p_j}| < 1$ for all $j$ and $\boldsymbol{b}\big((b_{p_j})\big) \leq  1$\,.
\end{enumerate}
\end{enumerate}
\end{theorem}
 For the proof recall the preceding remark and apply Theorems~\ref{polinomios},  \ref{Leonhard}, \ref{final}, and \ref{general}.

\subsection{Bohr's absolute convergence problem -- old art in new light}
\noindent In  what remains we would like to illustrate  that this characterization includes many  results on
the width of Bohr's strips, old and new ones,  as  special cases:

\begin{itemize}
\item
The Bohr-Bohnenblust-Hille theorem $S= \frac{1}{2}$ (see \eqref{BBHthm}) states  in terms of multipliers that
\begin{equation} \label{orkan}
\inf \big\{ \sigma\,\, \big|\,\, (1/n^\sigma) \text{ is an } \ell_1-\text{multiplier for } \mathcal{H}_\infty  \big\} = \frac{1}{2}\,.
\end{equation}
 Theorem~\ref{multiplier}-\textit{(\ref{m2ii})} determines all multiplicative $\ell_1$-multi\-plier for $\mathcal{H}_\infty$, and the  Bohr-Bohnenblust-Hille theorem  is a simple consequence: $S \le \frac{1}{2}$
since  for all $\varepsilon >0$ the  sequences $\big( p_k^{-\frac{1}{2}-\varepsilon} \big)$ belong to $ \mathbf{B}$, and $S \ge \frac{1}{2}$ since
$\big( p_k^{-\frac{1}{2}+\varepsilon} \big) \notin \mathbf{\overline{B}}\,.$ By \eqref{oje} this infimum from \eqref{orkan} is attained -- a result which in our setting can alternatively be deduced from
Theorem~\ref{multiplier}-\textit{(\ref{m2ii})} since $(p_k)^{-1/2} \in \mathbf{B}$.

 \item
 For Bohr strips of  $m$-homogeneous Dirichlet series  we by \eqref{BHthm1} have that $S_m = \frac{m-1}{2m}$. Again this result can be reformulated into a result on
 $\ell_1$-multi\-pliers for $\mathcal{H}^m_\infty$ of the type $(1/n^\sigma)$, and hence it can be easily deduced from the more general Theorem~\ref{multiplier}-\textit{(\ref{m1ii})}.
\item
Let now $1 \le p < \infty$.
It is known that each $\mathcal{H}_p$-Dirichlet series $D$  has an absolut convergence abscissa $\sigma_a(D) \leq 1/2$, and that this estimate is optimal:
\begin{align} \label{orkan2}
 \sup_{D \in \mathcal{H}_p} \sigma_a(D)= \frac{1}{2}\,.
\end{align}
 This is an $\mathcal{H}_p$-analog of \eqref{BBHthm} (or  equivalently \eqref{jetztdoch}) which can be found (implicitly) in \cite{Ba02} and (explicitly) in \cite[Theorem 1.1]{BaCaQu06}.
  After the following  reformulation in terms of $\ell_1$-multipliers for $\mathcal{H}_p$\,:
\begin{align} \label{orkan1}
\inf \big\{ \sigma\,\,\big|\,\, (1/n^\sigma) \text{ is an } \ell_1-\text{multiplier for } \mathcal{H}_p  \big\} = \frac{1}{2}\,,
\end{align}
  we obtain \eqref{orkan2}  as an immediate consequence of Theorem~\ref{multiplier}-\textit{(\ref{m2i})}.
 Note that here in contrast with \eqref{orkan} the infimum in \eqref{orkan1}
is not attained since $(p_k^{-1/2})_{k} \notin \ell_2$ (see also \cite{BaCaQu06} where this was observed for the first time).
\item Similarly we obtain $\sup_{D \in \mathcal{H}^m_\infty} \sigma_a(D) = \frac{m-1}{2m}  $ as a consequence of Theorem~\ref{multiplier}-\textit{(\ref{m1ii})}, and observe
that here  the infimum corresponding to \eqref{orkan1} is attained (see also \eqref{ojeje}).
\end{itemize}

\noindent \\bayart@math.univ-bpclermont.fr\\defant@mathematik.uni-oldenburg.de \\ frerick@uni-trier.de \\ maestre@uv.es \\ psevilla@mat.upv.es

\end{document}